\documentclass[reqno,draft]{amsart}
\usepackage{amsmath, amsthm, amssymb, dsfont}
\usepackage[a4paper]{geometry}					
\usepackage{graphicx}
\usepackage{tikz}
\usepackage{enumerate}

\usepackage{latexsym}
\usepackage{bbm}
\usepackage{mathtools}

\usepackage{comma}
\usepackage{cutwin}
\usepackage{pstricks}
\usepackage{pst-plot}
\usepackage{pgfplots}
\usepackage{float}

\numberwithin{equation}{section}

\newcommand{\defeq}{\vcentcolon=}
\newcommand{\eqdef}{=\vcentcolon}

\newcommand{\N}{\mathbb{N}}
\newcommand{\Z}{\mathbb{Z}}

\newcommand{\R}{\mathbb{R}}
\newcommand{\C}{\mathbb{C}}

\newcommand{\imag}{\mathrm{i}}
\newcommand{\Real}{\mathrm{Re}}

\newcommand{\inner}{\mathrm{int}}

\newcommand{\1}{\mathbbm{1}}

\newcommand{\<}{\langle}
\renewcommand{\>}{\rangle}

\newcommand{\x}{\mathsf{x}}
\newcommand{\y}{\mathsf{y}}
\newcommand{\p}{p_\lambda}
\newcommand{\q}{q_\lambda}

\newcommand{\Prm}{\mathrm{P}}
\newcommand{\Erm}{\mathrm{E}}
\newcommand{\Prmp}{\mathrm{P}_{\!p}}

\newcommand{\Prob}{\mathbb{P}}

\newcommand{\E}{\mathbb{E}}

\newcommand{\B}{\mathcal{B}}
\newcommand{\Cluster}{\mathcal{C}}

\newcommand{\disteq}{%
  \mathrel{\vbox{\offinterlineskip\ialign{%
    \hfil##\hfil\cr
    $\scriptscriptstyle\mathrm{law}$\cr
    \noalign{\kern.15ex}
    $=$\cr
}}}}

\newcommand{\distto}{%
  \mathrel{\vbox{\offinterlineskip\ialign{%
    \hfil##\hfil\cr
    $\scriptscriptstyle\mathrm{d}$\cr
    \noalign{\kern.15ex}
    $\to$\cr
}}}}

\newcommand{\Probto}{%
  \mathrel{\vbox{\offinterlineskip\ialign{%
    \hfil##\hfil\cr
    $\scriptscriptstyle\!\Prob$\cr
    \noalign{\kern.15ex}
    $\to$\cr
}}}}

\newcommand{\distleq}{\preccurlyeq}

\newcommand{\F}{\mathcal{F}}
\newcommand{\G}{\mathcal{G}}

\newcommand{\dx}{\mathrm{d}\mathit{x}}

\newcommand{\vel}{\overline{\mathrm{v}}}

\newcommand{\lambdacrit}{\lambda_{\mathrm{c}}}
\newcommand{\pesc}{\mathit{p}_{\mathrm{esc}}}
\newcommand{\bfnull}{\mathbf{0}}
\newcommand{\omegaprune}{\omega^{\mathrm{p}}}
\newcommand{\cp}{c^{\mathrm{p}}}

\newcommand{\domega}{\mathrm{d}\omega}
\usepackage{changes}

\theoremstyle{plain}
\newtheorem{theorem}{Theorem}[section]
\newtheorem{lemma}[theorem]{Lemma}
\newtheorem{proposition}[theorem]{Proposition}
\theoremstyle{remark}

\theoremstyle{definition}
\newtheorem{definition}[theorem]{Definition}

\begin{document}

\title[Random walk on $1$D percolation cluster at critical bias]{The speed
of critically biased random walk in a one-dimensional percolation model}

\authors{Jan-Erik L\"ubbers\footnote{Technische Universit\"at Darmstadt, Germany. Email: luebbers@mathematik.tu-darmstadt.de}
\and
Matthias Meiners\footnote{Universit\"at Innsbruck, Austria. Email: matthias.meiners@uibk.ac.at}}

\begin{abstract}
We consider biased random walks in a one-dimensional percolation model.
This model goes back to Axel\-son-Fisk and H\"aggstr\"om
and exhibits the same phase transition as biased random walk on the
infinite cluster of supercritical Bernoulli bond percolation on $\Z^d$,
namely, for some critical value $\lambdacrit>0$ of the bias,
it holds that the asymptotic linear speed $\vel$ of the walk is strictly
positive if the bias $\lambda$ is strictly smaller than $\lambdacrit$,
whereas $\vel=0$ if $\lambda \geq \lambdacrit$.

We show that at the critical bias $\lambda = \lambdacrit$, the
displacement of the random walk from the origin is of order $n/\log n$.
This is in accordance with simulation results by Dhar and Stauffer for
biased random walk on the infinite cluster of supercritical bond
percolation on $\Z^d$.

Our result is based on fine estimates for the tails of suitable
regeneration times.
As a by-product of these estimates we also obtain the order of
fluctuations of the walk in the sub-ballistic and in the ballistic, 
nondiffusive phase.
\smallskip

\\{\bf Subclass:} MSC 60K37 $\cdot$ MSC 82B43
\end{abstract}

\maketitle

\section{Introduction and main results}

\subsection{Introduction}

In the physics literature, biased random walk on a percolation cluster 
is considered as a model for transport in an inhomogeneous medium.
The mathematically rigorous study of biased random walk on the infinite 
cluster of supercritical Bernoulli bond percolation on $\Z^d$ was 
initiated in two parallel papers
by Berger, Gantert and Peres \cite{Berger+Gantert+Peres:2003}, and 
Sznitman \cite{Sznitman:2003}.
Both papers establish an interesting phenomenon,
namely, if the strength of the bias is positive but small, then the 
linear speed of the walk is positive,
whereas it is zero if the strength of the bias is sufficiently large.
The sharpness of the phase transition, which had been conjectured in 
the physics literature by Barma and Dhar \cite{Barma+Dhar:1983}, 
remained open.
An indication for the validity of the conjecture was provided by work of 
Lyons, Pemantle and Peres \cite{Lyons+Pemantle+Peres:1996},
who had shown that there is an analogous phase transition for the 
simpler model of biased random walk on a Galton--Watson tree with leaves,
and that the phase transition in this model is indeed sharp. Moreover, 
the result of Lyons, Pemantle and Peres includes the statement that the 
speed at the critical bias equals zero.
A rigorous proof of the sharpness of the phase transition for biased 
random walk on the infinite cluster of supercritical Bernoulli bond 
percolation on $\Z^d$
was eventually given by Fribergh and Hammond \cite{Fribergh+Hammond:2014}.
In this paper, the authors conjecture that the speed at the critical 
bias equals zero.
What is more, in the physics literature, it was conjectured by Dhar and Stauffer 
\cite{Dhar+Stauffer:1998}
that the displacement of the critically biased random walk from the 
origin at time $n$
(in the direction of the bias) is of the order $n/\log n$.

In the present paper, we shall prove this conjecture for biased random 
walk on a one-dimensional percolation cluster.
This model was created by Axelson-Fisk and H\"aggstr\"om
in \cite{Axelson-Fisk+H"aggstr"om:2009b,Axelson-Fisk+H"aggstr"om:2009}
to be simpler than biased random walk on the infinite cluster of
supercritical Bernoulli bond percolation on $\Z^d$,
but to display qualitatively similar phenomena.
Moreover, the initial hope might have been to construct a model that
is even amenable to explicit calculations.
And indeed, Axelson-Fisk and H\"aggstr\"om 
\cite{Axelson-Fisk+H"aggstr"om:2009b}
were able to express the critical bias as an elementary function of a 
percolation parameter of the model.
However, more complicated quantities such as the asymptotic linear speed 
as a function of the percolation parameter and the strength of the bias
withstood explicit calculation so far.

Our proof of the fact that the displacement of the critically biased 
random walk at time $n$ is of the order $n / \log n$
is based on refined estimates for the tails of suitable regeneration 
times that were introduced and studied in a joint paper of the second author
with Gantert and M\"uller \cite{Gantert+al:2018}. Our bounds on the 
tails of the regeneration times do not only hold for the critical bias
but for a large range of biases including the whole sub-ballistic and 
the ballistic, nondiffusive phase.
This allows us to deduce the order of the fluctuations of the walk in 
these phases.
Our result on the fluctuations of the biased random walk in the 
sub-ballistic phase
parallels the corresponding results for biased random walk on a 
Galton--Watson tree with leaves
due to Ben Arous et al.\ \cite{BenArous+Fribergh+Gantert+Hammond:2012}
and is more precise than the corresponding result for random walk on the 
infinite cluster of supercritical Bernoulli bond percolation on $\Z^d$
obtained in \cite{Fribergh+Hammond:2014}.

\subsection{Model description}

In this section, we give a brief introduction to the model and review some
results that are required for the formulation of our main results.

Consider the \emph{ladder graph} $G = (V,E)$
with vertex set $V = \Z \times \{0,1\}$ and edge set $E = \{ \<u,v\> \in
V^2: |u-v| = 1\}$ where $|\cdot|$ denotes the usual Euclidean norm on $\R^2$.
If $v = (x,y) \in V$, we write $\x(v)=x$ and $\y(v)=y$,
and call $x$ and $y$ the $\x$- and $\y$-coordinate of $v$, respectively.

In a first step, we consider i.i.d.\ bond percolation with retention
parameter $p \in (0,1)$ on $G$,
i.e., each edge $e \in E$ is retained independently of all other edges
with probability $p$,
and deleted with probability $1-p$.
As usual, we call an edge $e\in E$ \emph{open} if it is retained and
\emph{closed} if it is deleted.
The state space of the percolation process is $\Omega = \{0,1\}^E$,
which we endow with the product $\sigma$-algebra $\F$.
The elements $\omega \in \Omega$ are called \emph{configurations}.
We interpret $\omega(e)=1$ for $\omega \in \Omega$ and $e \in E$
as the edge $e$ being open in the configuration $\omega$.
A path between $u,v \in V$ is a finite sequence $P = (e_1,\ldots,e_n)$
of edges $e_1 = \< u_0,u_1 \>,\ldots,e_n = \<u_{n-1},u_n\> \in E$ with
$u_0=u$ and $u_n = v$.
The path $P$ is called \emph{open} if $\omega(e_k)=1$ for $k=1,\ldots,n$.
Let $\Omega_{N_1,N_2}$ be the event that there exists an open path
connecting a vertex
with $\x$-coordinate $-N_1$ to a vertex with $\x$-coordinate $N_2$,
and let $\Prm_{p,N_1,N_2}$ be the probability measure on $(\Omega,\F)$
arising from conditioning i.i.d.\ bond percolation with parameter $p$
on the event $\Omega_{N_1,N_2}$.
Then $\Prm_{p,N_1,N_2}$ converges weakly as $N_1,N_2 \to \infty$ to a
probability measure $\Prmp^*$
on $(\Omega,\F)$.

\begin{proposition}[Theorem 2.1 and Corollary 2.2 in
\cite{Axelson-Fisk+H"aggstr"om:2009}]
For any $p \in(0,1)$, as $N_1,N_2 \to \infty$, the probability measures
$\Prm_{p,N_1,N_2}$ converge weakly to
a translation invariant probability measure $\Prmp^*$ on $(\Omega,\F)$
satisfying $\Prmp^*(\Omega^*) = 1$
where $\Omega^* = \bigcap_{N_1,N_2 \in \N} \Omega_{N_1,N_2}$ is the event
that a bi-infinite open path exists.
\end{proposition}

It is easily seen that $\Prmp^*$-almost surely (a.\,s.), there is a unique infinite
open cluster $\Cluster \subseteq V$
consisting of all vertices $v \in V$ which are connected via open paths to
vertices with arbitrary $\x$-coordinate.
We define $\Prmp(\cdot) \defeq \Prmp^*( \cdot | \bfnull  \in \Cluster)$
where $\bfnull \defeq (0,0)$.

Henceforth, we fix a parameter $p \in (0,1)$.
Most of the constants and objects defined below will depend on $p$,
but this will usually not figure in the notation.

After choosing an environment $\omega \in \{0,1\}^E$ according to $\Prmp$,
we define a random walk on $G$
with bias $\lambda \in \R$ as follows.
Let the conductances $(c(e))_{e \in E}$ be defined via
\begin{equation*}
c(\<u,v\>)	\defeq	e^{\lambda (\x(u) + \x(v))},	\quad	\<u,v\> \in E.
\end{equation*}
Then $(Y_n)_{n\in\N_0}$ is defined as the lazy random walk with conductances $(c(e))_{e \in E}$ on $\Cluster$
starting at $Y_0 \defeq \bfnull $.
More precisely, when at $u \in V$, the walk attempts to move to a neighbor
$v \in V$ in $G$
with probability proportional to $c(\<u,v\>)$.
The step is actually performed if $\omega( \<u,v\>)=1$, otherwise, the
walk stays put.
We denote the law of $(Y_n)_{n \in \N_0}$ on $(V^{\N_0}, \G)$ by
$P_{\omega,\lambda}$,
where $\G$ is the product $\sigma$-algebra on $V^{\N_0}$.
Further, we write $P_{\omega,\lambda}^v$ for the law of the Markov chain
with the same transition probabilities but with start at $v \in V$.
By the symmetry of the law of $\omega$ it suffices to consider the case $\lambda>0$.

The distribution $P_{\omega,\lambda}$ is the \emph{quenched} law of
$(Y_n)_{n \in \N_0}$ (given $\omega$).
The corresponding \emph{annealed} law is obtained by averaging the
quenched laws over $\omega \in \Omega$ using $\Prmp$.
Formally, we define the probability measure $\Prob$ on $\{0,1\}^E \times
V^{\N_0}$ as follows.
For $A \in \F, B\in\G$ set
\begin{equation}	\label{eq:Prob}
\Prob(A \times B) \defeq \int_A P_{\omega,\lambda}(B) \, \Prmp(\domega).
\end{equation}
Notice that $\Prob$ depends on $\lambda$ and $p$ even though both
parameters do not figure in the notation.
For $\lambda > 0$, under $\Prob$,
the walk $(Y_n)_{n \in \N_0}$ is transient and there exists a critical value
$\lambdacrit$ for the bias such that
$X_n \defeq \x(Y_n)$ has positive linear speed if $\lambda < \lambdacrit$,
and zero linear speed if $\lambda \geq \lambdacrit$.
This comes from the fact that the larger the bias,
the more time the walk needs to leave dead-ends in the direction of the bias.

\begin{proposition}[Proposition 3.1 and Theorem 3.2 in
\cite{Axelson-Fisk+H"aggstr"om:2009b}]	\label{Prop:transience of Y_n}
Fix $\lambda>0$. The walk $(Y_n)_{n\in\N_0}$ is $\Prob$-a.\,s.\ transient,
and $\lim_{n\to\infty} \frac{ X_n }{n} = \vel(\lambda)$ $\Prob$-a.\,s.\ with
\begin{equation*}
\vel(\lambda)	=	\begin{cases}
				> 0	& \text{for } \lambda \in (0, \lambdacrit),	\\
				= 0	& \text{for } \lambda \geq \lambdacrit
				\end{cases}
\end{equation*}
where $\lambdacrit = \frac{1}{2} \log \big( 2/\big(
1+2p-2p^2-\sqrt{1+4p^2-8p^3+4p^4 }\big) \big)$.
\end{proposition}

Existence of a critical value for the bias has been proven in similar
models, e.\,g.\ in \cite{Lyons+Pemantle+Peres:1996}
for biased random walks on Galton-Watson trees
and in \cite{Fribergh+Hammond:2014} for biased random walk on the
supercritical percolation cluster in $\Z^d$.
In the present setting, $\lambdacrit$ is given as an elementary function of $p$.

\subsection{Main results}

The main results of this paper concern the speed of biased random walk in
the sub-ballistic regime.
If the bias is critical ($\lambda = \lambdacrit$),
$X_n$ is of order $n/\log n$.
This is in alignment with simulation results for biased random walk on the
infinite cluster
of supercritical bond percolation in $\Z^d$ in \cite{Dhar+Stauffer:1998}.

\begin{theorem}	\label{Thm:critical speed}
In the case $\lambda=\lambdacrit$, there exist constants $0< a < b <
\infty$ such that
\begin{equation*}
\lim_{n \to \infty} \Prob \big(\tfrac{X_n}{n/\log n} \in [a,b] \big) = 1.
\end{equation*}
\end{theorem}

We prove this theorem from fine estimates for the tails of suitable
regeneration times to be introduced below.
Less accurate estimates for the tails of these regeneration times derived in
\cite{Gantert+al:2018}
revealed a second phase transition at $\lambda = \lambdacrit/2$, namely,
a central limit theorem for $(X_n)_{n \in \N_0}$ with square-root scaling
holds
if and only if $\lambda < \lambdacrit/2$,
see \cite[Theorem 2.6]{Gantert+al:2018}.
Our tail estimates also give control over the fluctuations of $(X_n)_{n \in \N_0}$
in the remaining parameter range $\lambda \in [\lambdacrit/2,\infty)$.

Throughout the paper, we write
\begin{equation*}
\alpha	\defeq	\lambdacrit/\lambda.
\end{equation*}

\begin{theorem}	\label{Thm:noncritical regimes}
Suppose that $\lambda \geq \lambdacrit/2$, $\lambda \neq \lambdacrit$.
\begin{enumerate}[(a)]\itemsep1pt
	\item
		Let $\lambda = \lambdacrit/2$, i.e., $\alpha = 2$.
		Then the laws of $\big(\frac{X_n - n \vel}{\sqrt{n \log n}}\big)_{n \geq 2}$ under $\Prob$ are tight.
	\item
		Let $\lambda \in (\frac{\lambdacrit}{2}, \lambdacrit)$, i.e., $\alpha \in (1,2)$.
		Then the laws of $\big(\frac{X_n - n \vel}{n^{1/\alpha}}\big)_{n\in\N}$ under $\Prob$ are tight.
	\item
		Let $\lambda > \lambdacrit$, i.e., $\alpha \in (0,1)$.
		Then the laws of $\big(\frac{X_n}{n^\alpha}\big)_{n\in\N}$ under $\Prob$ are tight.
\end{enumerate}
\end{theorem}

In all three cases covered by Theorem \ref{Thm:noncritical regimes},
we do not expect that tightness can be strengthened to convergence in
distribution
due to a lack of regular variation of the tails of the regeneration times,
see Lemma \ref{Lem:annealed tail estimate for single trap} and the proof thereof.
Instead, we expect only convergence along certain subsequences as found for biased
random walk on Galton-Watson trees, cf.\ \cite{BenArous+Fribergh+Gantert+Hammond:2012}.
We refrain from further investigating this phenomenon,
as our main goal in this paper is to derive the speed of biased random
walk at the critical bias.

We continue with an overview of the organization of the paper.
In Section \ref{sec:regeneration}, we introduce regeneration points and times that go back to \cite{Gantert+al:2018}.
We review known results about the regeneration points and times and state our main technical result, Proposition \ref{Prop:tail estimate},
which provides the precise order of the tails of the regeneration times.
Based on these tail bounds, we prove the main results in Section \ref{sec:proofs of main results}.
Section \ref{sec:tail bounds} is devoted to the proof of Proposition \ref{Prop:tail estimate}.
Finally, in Appendix \ref{sec:renewal theory}, we provide an auxiliary result from renewal theory.

\section{Regeneration points and times}	\label{sec:regeneration}

We use the decomposition of the percolation cluster at regeneration points
from \cite{Gantert+al:2018}.
Regeneration points are defined in two steps.
Given a configuration $\omega \in \Omega$,
a vertex $v = (\x(v),0) \in V$ is called a \emph{pre-regeneration point}
if $v \in \Cluster$ and $(\x(v),1)$ is an isolated vertex in $\omega$,
that is,
all three edges adjacent to $(\x(v),1)$ are closed in $\omega$.

\begin{lemma}[Lemma 5.1 and Corollary 5.2 in
\cite{Axelson-Fisk+H"aggstr"om:2009b}]	\label{Lem:number of regeneration
points}
With $\Prmp$-probability one, there exist infinitely many pre-regeneration
points both left and right of the origin.
\end{lemma}

We enumerate the pre-regeneration points in $\omega$ by
$\ldots, R^{\mathrm{pre}}_{-1}, R^{\mathrm{pre}}_0,
R^{\mathrm{pre}}_1,\ldots$
such that $\x(R^{\mathrm{pre}}_{-1}) < 0 \leq \x(R^{\mathrm{pre}}_0)$ and
$\x(R^{\mathrm{pre}}_n) <  \x(R^{\mathrm{pre}}_{n+1})$ for all $n \in \Z$.

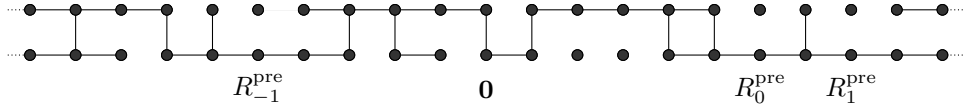
\begin{figure}[htbp]
\begin{center}
\pgfmathsetseed{845133}
\begin{tikzpicture}[thin, scale=0.6,-,
                   shorten >=0pt+0.5*\pgflinewidth,
                   shorten <=0pt+0.5*\pgflinewidth,
                   every node/.style={circle,
                                      draw,
                                      fill          = black!80,
                                      inner sep     = 0pt,
                                      minimum width =4 pt}]

\def \p {0.5}

\foreach \x in {-10,-9,-8,-7,-6,-5,-4,-3,-2,-1,0,1,2,3,4,5,6,7,8,9,10}
\foreach \y in {0,1}
    \node at (\x,\y) {};

\foreach \x in {-10,-9,-8,-7,-6,-5,-4,-3,-2,-1,0,1,2,3,4,5,6,7,8,9}{
\foreach \y in {0,1}{
    \pgfmathparse{rnd}
    \let\dummynum=\pgfmathresult
    \ifdim\pgfmathresult pt < \p pt\relax \draw (\x,\y) -- (\x+1,\y);\fi
  }}

\foreach \x in {-10,-9,-8,-7,-6,-5,-4,-3,-2,-1,0,1,2,3,4,5,6,7,8,9,10}{
\foreach \y in {0}{
    \pgfmathparse{rnd}
    \let\dummynum=\pgfmathresult
    \ifdim\pgfmathresult pt < \p pt\relax \draw (\x,\y) -- (\x,\y+1);\fi
  }}

	\draw[densely dotted] (-10.5,1) -- (-10,1);
	\draw[densely dotted] (-10.5,0) -- (-10,0);
	\draw (-8,1) -- (-7,1);
	\draw (-3,1) -- (-2,1);
	\draw (1,1) -- (2,1);
	\draw (6,0) -- (7,0);
	\draw (9,0) -- (10,0);
	\draw[densely dotted] (10,0) -- (10.5,0);
	\draw[densely dotted] (10,1) -- (10.5,1);
	\node[draw=none,fill=none] at (0,-0.75) {$\mathbf{0}$};
	\node[draw=none,fill=none] at (6,-0.75) {$R^{\mathrm{pre}}_{0}$};
	\node[draw=none,fill=none] at (8,-0.75) {$R^{\mathrm{pre}}_{1}$};

	\draw (9,1) -- (10,1);
	\draw[white, thick] (-5,1) -- (-4,1);
	\foreach \x in {-10,-9,-8,-7,-6,-5,-4,-3,-2,-1,0,1,2,3,4,5,6,7,8,9,10}
\foreach \y in {0,1}
    \node at (\x,\y) {};
    \node[draw=none,fill=none] at (-5,-0.75) {$R^{\mathrm{pre}}_{-1}$};
\end{tikzpicture}
\caption{Pre-regeneration points close to the origin}
\end{center}
\end{figure}
The pre-regeneration points can be used to decompose the percolation
cluster into independent pieces.
For $a,b \in \Z$,
we denote the subgraph of $\omega$ with vertex set $V_{[a,b)} \defeq \{v
\in V: a \leq \x(v) \leq b\}$
and edge set $E_{[a,b)} \defeq \{e=\<u,v\> \in E:\, u,v \in V_{[a,b)},\,
\x(u) \wedge \x(v) < b,\, \omega(e)=1\}$
by $[a,b)$ and call $[a,b)$ a \emph{piece} or \emph{block} (of $\omega$).
We then define
\begin{center}
$\omega_n \defeq [\x(R_{n-1}^{\mathrm{pre}}),
\x(R_n^{\mathrm{pre}})),	\quad	n \in \Z.$
\end{center}
Using this definition, we may introduce the cycle-stationary percolation
law $\Prmp^\circ$.

\begin{definition}	\label{def:cycle-stationary percolation law}
The \emph{cycle-stationary percolation law} $\Prmp^\circ$
is defined to be the unique probability measure on $(\Omega,\F)$ such that
the cycles $\omega_n$, $n \in \Z$ are i.i.d.\ under $\Prmp^\circ$
with each $\omega_n$ having the same law under $\Prmp^\circ$
as $\omega_1$ under $\Prmp^*$, and such that $R^{\mathrm{pre}}_0=\bfnull$.
\end{definition}

We write $\Prob^\circ$ for the annealed law of the biased random walk and
the percolation configuration
when the latter is drawn using $\Prmp^\circ$ instead of $\Prmp$.
To be more precise, $\Prob^\circ$ is defined as $\Prob$ in
\eqref{eq:Prob}, but with $\Prmp$ replaced by $\Prmp^\circ$.

\begin{definition} \label{def:regeneration point}
We call a $v \in V$ with $\x(v) \geq 0$ \emph{regeneration point} if
\begin{enumerate}[1.]	\itemsep0pt
	\item	it is a pre-regeneration point and
	\item	the random walk $(Y_n)_{n \in \N_0}$ visits $v$ exactly once.
\end{enumerate}
\end{definition}

It follows from the discussion in Section 4 of \cite{Gantert+al:2018} that
there are infinitely many regeneration points to the right of $\bfnull$.
We set $R_0 \defeq \bfnull$ and, for $n \in \N$, define $R_n$ to be the
first regeneration point to the right of $R_{n-1}$.
Thus, $\rho_{n-1} < \rho_{n}$ for all $n \in \N$
where $\rho_n \defeq \x(R_n)$, $n \in \N_0$.
Furthermore, let $\tau_0 \defeq 0$ and
\begin{equation*}
\tau_n \defeq \inf \{k\in\N_0: Y_k = R_n \}, 	\qquad	n \in \N.
\end{equation*}
For $n \geq 1$, $\tau_n$ is the unique time at which the $n$th
regeneration point $R_n$
is visited by the walk $(Y_k)_{k \in \N_0}$.
In particular, $0 = \tau_0 < \tau_1 < \ldots$\,.
We call $\tau_n$ the $n$th \emph{regeneration time}.
The following assertions are known from \cite{Gantert+al:2018}
about the regeneration times and points.

\begin{lemma}[Lemmas 4.1 and 4.2, Proposition 4.3 in
\cite{Gantert+al:2018}]	\label{Lemma:known results rho and tau}
Fix $\lambda > 0$.
\begin{enumerate}[(a)]	\itemsep1pt
	\item
		Under $\Prob$, the pairs $(\tau_{n+1}\!-\!\tau_n, \rho_{n+1}\!-\!\rho_n
)$, $n\in\N$ are i.i.d.\ and
		independent of $(\tau_1,\rho_1)$,
		and
		\begin{equation*}
		\Prob((\tau_{2}\!-\!\tau_1, \rho_{2}\!-\!\rho_1) \in \cdot) =
\Prob^\circ((\tau_1,\rho_1) \in \cdot | Y_n \not = \bfnull \text{ for
all } n \geq 1).
		\end{equation*}
	\item
		There exists some $\delta>0$ such that $\E[e^{\delta(\rho_2 - \rho_1)}]
< \infty$.
	\item
		It holds that $\E[(\tau_2 - \tau_1)^\kappa] < \infty$ if and only if
$\kappa < \alpha = \lambdacrit/\lambda$.
	\item
		The ballistic speed satisfies $\vel(\lambda) = \E[\rho_2 -
\rho_1]/\E[\tau_2 - \tau_1]$.
\end{enumerate}
\end{lemma}

Lemma \ref{Lemma:known results rho and tau}(c) indicates that
$\Prob(\tau_2 - \tau_1 \geq n)$ is roughly of the order $n^{-\alpha}$ as
$n \to \infty$.
We give a more precise statement in the following proposition.

\begin{proposition}	\label{Prop:tail estimate}
For any $\lambda > \log(2)/2$, in particular for $\lambda\geq\lambdacrit/2$, there exist constants $0 < c \leq d <\infty$
(depending on $p$ and $\lambda$)
such that, for all $n \in \N$,
\begin{equation*}
c n^{-\alpha} \leq	\Prob(\tau_2 - \tau_1 \geq n) \leq d n^{-\alpha}.
\end{equation*}
and
\begin{equation*}
c n^{-\alpha} \leq	\Prob(\tau_1 \geq n) \leq d n^{-\alpha} \log n.
\end{equation*}
\end{proposition}

\begin{figure}[htbp]
\begin{center}
\begin{tikzpicture}[scale=1]
		\draw [help lines] (0,0) grid (5,4.1);
		\draw [-] (0,0) -- (5,0);
		\draw [->] (0,-0.1) -- (0,4.2);
		\draw [-] (5,-0.1) -- (5,0.1);
		\draw (0,0) node[below]{$0$};
		\draw (2.5,0) node[below]{$p$};
		\draw (5,0) node[below]{$1$};
		\draw (0,2*0.6931472/2) node[left]{$\frac{\log(2)}{2}$};
		\draw (0,2) node[left]{$1$};
		\draw (0,3) node[left]{$1.5$};
		\draw (0,4) node[left]{$2$};
		\draw [thick, domain=0.08:4.925, samples=128] plot(\x,
{ln(2/(1+2*(\x/5)-2*(\x/5)^2-sqrt(1+4*(\x/5)^2-8*(\x/5)^3+4*(\x/5)^4)))});
		\draw (2.5,1.7) node[above]{$\lambdacrit$};
		\draw [thick, domain=0.0012:4.999, samples=128] plot(\x,
{0.5*ln(2/(1+2*(\x/5)-2*(\x/5)^2-sqrt(1+4*(\x/5)^2-8*(\x/5)^3+4*(\x/5)^4)))});
		\draw (2.5,0.9) node[above]{$\frac{\lambdacrit}{2}$};
		\draw [thick, domain=0:5, samples=128] plot(\x, {2*ln(2)/2});
\end{tikzpicture}
\vspace{-0.3cm}
\caption{The figure shows $\lambdacrit$ and $\lambdacrit/2$ as functions
of $p$.
Our Proposition \ref{Prop:tail estimate} giving precise tail asymptotics
for the regeneration times applies for $\lambda > \log(2)/2$,
which is strictly smaller than $\lambdacrit/2$ for any $p\in(0,1)$.}
\label{fig:lambdacrit}
\end{center}
\end{figure}
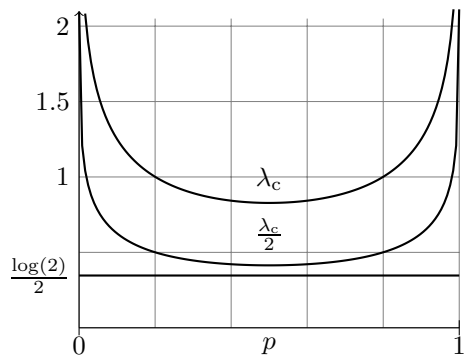

The bulk of the work in this paper is required to prove this proposition.
Before we turn to its proof,
we first demonstrate in the subsequent section how the main results of the
paper,
Theorems \ref{Thm:critical speed} and \ref{Thm:noncritical regimes},
can be derived from it.
The proofs of these theorems are generic in the sense that they
do not use the particular definition of $X_n$,
but will apply to any random walk $X_n$
for which there are regeneration points and times
satisfying the conclusions of Lemma \ref{Lemma:known results rho and tau}
and Proposition \ref{Prop:tail estimate}.

\section{Proofs of the main results}	\label{sec:proofs of main results}

\subsection{Preliminaries and notation}

For random variables $X$ and $Y$ with distribution functions $F$ and $G$,
respectively,
we say that $X$ is \emph{stochastically dominated} by $Y$, and write $X
\distleq Y$,
if $F(t) \geq G(t)$ for all $ t \in \R$.

Convergence in distribution of a sequence $(X_n)_{n\in\N}$ of random
variables towards a random variable $X$
is denoted $X_n \distto X$.
Analogously, convergence in probability of $X_n$ to $X$ under $\Prob$ is
denoted by
$X_n \Probto X$.

As usual, for sequences $a,b: \N \to [0,\infty)$, we write $a = o_n(b)$ or
$a_n = o(b_n)$ as $n \to \infty$
if for every $\epsilon > 0$
there is an $n_0 \in \N$ with $a_n \leq \epsilon b_n$ for all $n \geq n_0$.
We say that $a$ and $b$ are \emph{asymptotically equivalent} and write $a
\sim b$ or $a_n \sim b_n$ as $n \to \infty$
if $a_n,b_n > 0$ for all sufficiently large $n$ and $\lim_{n \to\infty}
a_n/b_n = 1$.
Finally, we write $a = \mathcal{O}_n(b)$ or $a_n = \mathcal{O}(b_n)$ as $n
\to \infty$
if there exists some $C>0$ such that $a_n \leq C b_n$ for all sufficiently
large $n$.

From Lemma \ref{Lemma:known results rho and tau}, we infer that the
$\tau_n$, $n \in \N$
are the points of a delayed renewal process on the integers.
The corresponding \emph{renewal counting process} and \emph{first passage
times},
we denote by
\begin{equation*}	\textstyle
k(n) \defeq \max \big\{k \in \N_0: \tau_k \leq n \big\}
\quad	\text{and}	\quad
\nu(n)	\defeq	k(n) + 1,
\end{equation*}
respectively, where $n \in \N_0$.
Notice that $k(n) = \max \{k \in \N_0: \rho_k \leq X_n\}$, $n \in \N_0$.

To infer Theorems \ref{Thm:critical speed} and \ref{Thm:noncritical
regimes} from Proposition \ref{Prop:tail estimate},
we shall choose a sequence $(\xi_k)_{k \in \N}$ of independent random variables the $\xi_k$, $k \geq 2$
are i.i.d., $\tau_2 \! - \! \tau_{1} \distleq \xi_2$ and
$\Prob(\xi_2 > n) \sim d n^{-\alpha}$ as $n \to \infty$ (where $d$ is
chosen as in Proposition \ref{Prop:tail estimate}).
Then the law of $\xi_2$ is in the (normal) domain of attraction of an
$\alpha$-stable law.
From general theory it then follows that, after a suitable renormalisation,
the first passage times $\nu_\xi(t) := \inf\big\{k \in \N: \sum_{i=1}^k \xi_i > t\}$ converge in distribution as $t \to \infty$.
This will imply tightness of the first passage times $\nu(n)$ with the
same renormalisation.
From this, we shall derive the dual results for $X_n$ which translate into
the statements of Theorems \ref{Thm:critical speed} and
\ref{Thm:noncritical regimes}.

\subsection{Proofs of Theorems \ref{Thm:critical speed} and
\ref{Thm:noncritical regimes}}

We begin with the proof of the results in the sub-ballistic regimes.

\begin{proof}[Proof of Theorem \ref{Thm:critical speed} and Theorem
\ref{Thm:noncritical regimes}(c)]
Suppose that $\lambda \geq \lambdacrit$
so that $\alpha \in (0,1]$.
Let $a_n \defeq n^{\alpha}$ if $\alpha \in (0,1)$ and $a_n \defeq n/\log n$ if
$\alpha = 1$.
For $n \in \N$, we have
\begin{equation}	\label{eqn:alpha<=1-sandwich}
\frac{ \rho_{k(n)} }{a_n} \leq \frac{X_n}{a_n} \leq
\frac{\rho_{\nu(n)}}{a_n} = \frac{\rho_{\nu(n)}}{\nu(n)}
\frac{\nu(n)}{a_n}.
\end{equation}
Since $\nu(n) \to \infty$ $\Prob$-a.\,s.\ as $n \to \infty$, Lemma
\ref{Lemma:known results rho and tau} and the strong law of large numbers
imply
\begin{equation*}
\frac{\rho_{\nu(n)}}{\nu(n)}
= \frac{1}{\nu(n)} \sum_{k=1}^{\nu(n)} (\rho_k \! - \! \rho_{k-1}) \to
\E[\rho_2 \! - \! \rho_1]	\quad	\Prob\text{-a.\,s.}
\end{equation*}
Using Proposition \ref{Prop:tail estimate}, we can find independent random
variables $\eta_k$, $k \in \N$ and $\xi_k$, $k \in \N$
such that $\eta_1,\eta_2,\ldots$ are i.i.d.\ and $\xi_2,\xi_3,\ldots$ are
i.i.d.\ and such that
$\eta_k \distleq \tau_{k} \! - \! \tau_{k-1} \distleq \xi_k$ for all $k
\in \N$
and
\begin{align*}
\Prob(\eta_1 > n ) \sim cn^{-\alpha}
\quad \text{and} \quad
\Prob(\xi_2 > n ) \sim dn^{-\alpha}
\qquad	\text{as } n \to \infty.
\end{align*}
Further, we may choose $\xi_1$ independent of $\xi_2,\xi_3,\ldots$
such that $\Prob( \xi_1 > n ) \sim dn^{-\alpha}\log n$ as $n \to \infty$.
We set $\nu_\eta(n) \defeq \inf \{k \in \N: \sum_{i=1}^k \eta_i > n\}$
and $\nu_\xi(n) \defeq \inf \{k \in \N: \sum_{i=1}^k \xi_i > n \}$.
Then it holds that $\nu_\xi(n) \distleq \nu(n) \distleq \nu_\eta(n)$ for
all $n \in \N_0$.
Furthermore,
Theorem 3a in \cite{Bingham:1972} says that there is an $\alpha$-stable
subordinator $(Y_\alpha(t))_{t \geq 0}$
with Laplace exponent $\log \E[\exp(-s Y_\alpha(t))] = -ts^\alpha$ for
$s,t \geq 0$ such that
\begin{equation}	\label{eqn:conv to X_alpha}
a_n^{-1} \nu_\eta(n) \distto c_\eta X_\alpha
\quad \text{and} \quad
a_n^{-1} \nu_\xi(n) \distto c_\xi X_\alpha
\end{equation}
where $X_\alpha = \sup \{t \geq 0: Y_\alpha(t) \leq 1 \}$ and $0 < c_\xi \leq
c_\eta < \infty$.
(Notice that other than in \cite{Bingham:1972}, here we allow $\xi_1$ to
have a distribution different than that of $\xi_2,\xi_3,\ldots$,
but the contribution of the first step vanishes as $n \to \infty$.)
The difference of upper and lower bound in \eqref{eqn:alpha<=1-sandwich}
satisfies
\begin{equation}	\label{eqn:diff lower and upper bound in
eqn:alpha<=1-sandwich}
\frac{\rho_{\nu(n)}}{a_n} - \frac{\rho_{k(n)}}{a_n}	 = \frac{\rho_{\nu(n)}
- \rho_{\nu(n)-1}}{\nu(n)} \frac{\nu(n)}{a_n}
\Probto 0	\quad	\text{as } n \to \infty.
\end{equation}
Indeed, the first factor on the right-hand side converges to $0$
$\Prob$-a.\,s.\ as $n \to \infty$
due to Lemma \ref{Lemma:known results rho and tau}(b)
and \cite[Theorem 1.2.3(i)]{Gut:2009}
while the family of laws corresponding to the second factor are tight by \eqref{eqn:conv to X_alpha}.
Consequently, the difference in \eqref{eqn:diff lower and upper bound in
eqn:alpha<=1-sandwich} converges to 0 in distribution
and thus in $\Prob$-probability.

Now suppose $\alpha = 1$.
Then $Y_1(t) = t$ $\Prob$-a.\,s.\ and hence $X_1 = 1$ $\Prob$-a.\,s.
The convergence in \eqref{eqn:conv to X_alpha} thus is in fact convergence
in probability.
This completes the proof of Theorem \ref{Thm:critical speed}.

Finally, if $0 < \alpha < 1$,
then \eqref{eqn:conv to X_alpha} and $\nu_\xi(n) \distleq \nu(n) \distleq
\nu_\eta(n)$ for all $n \in \N_0$
imply that the family of laws of $(\nu(n)/n^\alpha)_{n \in \N}$ is tight.
From \eqref{eqn:alpha<=1-sandwich} and \eqref{eqn:diff lower and upper
bound in eqn:alpha<=1-sandwich}
we conclude that this carries over to the family of laws of
$(X_n/n^\alpha)_{n \in \N}$.
\end{proof}

We now turn to the proof of the main results for ballistic, nondiffusive biases.

\begin{proof}[Proof of Theorem \ref{Thm:noncritical regimes}(a) and (b)]
We prove (a) and (b) simultaneously.
Let $a_n \defeq n^{1/\alpha}$ in the case $\alpha \in (1,2)$
and $a_n \defeq \sqrt{n \log n}$ if $\alpha = 2$.
For $n \in \N$, we have
\begin{equation*}
\frac{\rho_{k(n)} - n \vel}{a_n} \leq \frac{X_n - n \vel}{a_n}
\leq \frac{\rho_{\nu(n)} - n \vel}{a_n}.
\end{equation*}
By the strong law of large numbers, $\nu(n)/n \to 1/\E[\tau_2-\tau_1] \in
(0,\infty)$ $\Prob$-a.\,s.
This together with Lemma \ref{Lemma:known results rho and tau} and
\cite[Theorem 1.2.3(i)]{Gut:2009}
implies $(\rho_{\nu(n)} \! - \! \rho_{k(n)})/ a_n \to 0$ $\Prob$-a.\,s.
On the other hand,
\begin{equation*}
\frac{\rho_{\nu(n)} - n \vel}{a_n}
= \frac{ \rho_{\nu(n)} - \nu(n) \E[\rho_2\!-\!\rho_1]}{ a_n } + \frac{
\nu(n) \E[\rho_2\!-\!\rho_1] - n \vel}{a_n}.
\end{equation*}
The first summand converges to $0$ $\Prob$-a.\,s.\ by \cite[Theorem
1.2.3(ii)]{Gut:2009} if $\alpha \in (1,2)$
and it converges to $0$ in $\Prob$-probability by \cite[Theorem
1.3.1]{Gut:2009} if $\alpha=2$.
It thus remains to check tightness of the family of laws of
\begin{equation*}
\frac{\nu(n)\E[\rho_2 \! - \! \rho_1] - n \vel}{a_n}
= \E[\rho_2 \! - \! \rho_1] \frac{\nu(n) - n/\E[\tau_2 \! - \!
\tau_1]}{a_n},		\quad	n\in\N.
\end{equation*}
For this, uniform integrability of the sequence $(a_n^{-1}(\nu(n) -
n/\E[\tau_2-\tau_1]))_{n\in\N}$
is sufficient. It thus remains to refer to Proposition \ref{Prop:tail
estimate} and Proposition \ref{Prop:uniform integrability of renewal counting
process} in the Appendix.
\end{proof}

\section{Proof of the tail estimate for regeneration times}	\label{sec:tail bounds}

It remains to prove the tail estimate for regeneration times, Proposition
\ref{Prop:tail estimate}.
This will be done in this section.
We begin with the analysis of traps, which will almost immediately result
in a proof of the lower bound in Proposition \ref{Prop:tail estimate}.

\subsection{Traps and biased random walk on a line segment}

As for biased random walk on the supercritical percolation cluster, the
slowdown in the model considered here is due to traps.
These are dead-end regions stretching in the direction of the bias.
For (conditional) percolation on the ladder graph,
this boils down to parallel finite open horizontal line segments with no vertical
connections.

To give a formal definition of a trap, we introduce some notation.
For a vertex $u \in V$, we write $u'$ for $(\x(u),1-\y(u))$.
Further, if $e = \< u,v \> \in E$, we let $e' \defeq \< u',v' \>$.
In particular, $e=e'$ if $e$ is a vertical edge, and $e'$ is the
horizontal edge parallel to $e$
if $e$ is a horizontal edge.
Now we define a \emph{trap} (in $\omega$) to be an open path $P =
(e_1,\ldots,e_m)$
of length $m \in \N$
with edges $e_1 = \< u_0,u_1 \>,\ldots,e_m = \<u_{m-1},u_m\> \in E$ such that
\begin{enumerate}[1.]\itemsep0pt
	\item		$\x(u_{k}) = \x(u_{k-1})+1$ and $\y(u_k) = \y(u_{k-1})$ for
$k=1,\ldots,m$;
	\item		the edges $\<u_0,u_0'\>$ and $e_k'$, $k=1,\ldots,m$ are open (in
$\omega$);
	\item		the edge $\<u_m,u_{m+1}\>$ is closed (in $\omega$) where $u_{m+1}
= (\x(u_m)+1,\y(u_m))$;
	\item		all vertical edges $\<u_k,u_k'\>$ for $k=1,\ldots,m$ are closed
(in $\omega$).
\end{enumerate}
Here, $m$ is called the \emph{length of the trap}, $u_0$ is called the
\emph{trap entrance} and $u_m$ is called the \emph{bottom of the trap}.
\vspace{-0.5cm}
\begin{center}
\begin{tikzpicture}[thin, scale=0.6,-,
                   shorten >=2pt+0.5*\pgflinewidth,
                   shorten <=2pt+0.5*\pgflinewidth,
                   every node/.style={circle,
                                      draw,
                                      fill          = black!80,
                                      inner sep     = 0pt,
                                      minimum width =4 pt}]
\path[draw] 
       node at (0,0) {}
       node at (0,1) {}
       node at (1,0) {}
       node at (1,1) {}
       node at (2,0) {}
       node at (2,1) {}
       node at (3,0) {}
       node at (3,1) {}
       node at (4,0) {}
       node at (4,1) {}
       node at (5,0) {}
       node at (5,1) {}
        ;

    \draw (0,0) -- (0,1) ;
    \draw (0,1) -- (1,1) ;
    \draw (0,0) -- (1,0) ;
    \draw (1,1) -- (2,1) ;
    \draw (1,0) -- (2,0) ;
    \draw (2,0) -- (3,0) ;
    \draw (2,1) -- (3,1);
    \draw (3,0) -- (4,0) ;
     \draw (3,1) -- (4,1);
    \draw (4,1) -- (5,1);

\begin{scope}[dashed]  

      \draw (-0.5,0) -- (-0,0)  ;
      \draw (-0.5,1) -- (-0,1)  ;

      \draw (5,0) -- (5,1)  ;
      \draw (5,0) -- (5.5,0)  ;
      \draw (5,1) -- (5.5,1);

\end{scope}

\begin{scope}   [->,shorten >=4pt+0.5*\pgflinewidth,
                   shorten <=8pt+0.5*\pgflinewidth,
]
	\draw (6,2) -- (5,1);
	\node[draw=none,fill=none] at (6,2) {$\mbox{trap end}$};
	\draw (-1,-1) -- (0,0);
	\node[draw=none,fill=none] at (-1,-1) {$\mbox{trap entrance}$};
	\draw (5,-1) -- (4,0);
	\node[draw=none,fill=none] at (5,-1) {$\mbox{bottom of the trap}$};
 \end{scope}
\end{tikzpicture}
\end{center}
\vspace{-1cm}
The piece $[\x(u_0),\x(u_{m+1}))$ is called (the corresponding) \emph{trap
piece}.

We define the backbone $\B$ to be the subgraph of the infinite
cluster $\Cluster$
obtained by deleting from $\Cluster$ all edges and all vertices in traps
except the trap entrance vertices.
Clearly, $\B$ is connected and contains all pre-regeneration points.
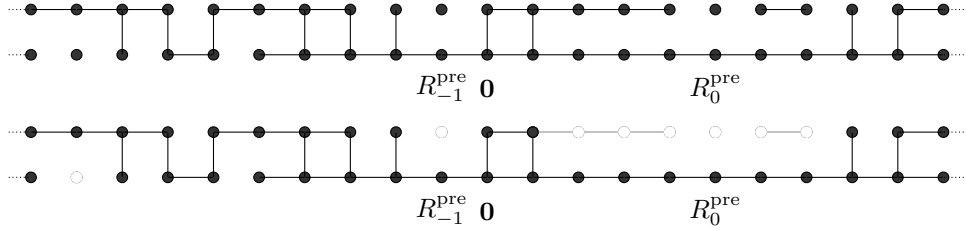
\begin{figure}[h]
\begin{center}
\pgfmathsetseed{216642}
\begin{tikzpicture}[thin, scale=0.6,-,
                   shorten >=0pt+0.5*\pgflinewidth,
                   shorten <=0pt+0.5*\pgflinewidth,
                   every node/.style={circle,
                                      draw,
                                      fill          = black!80,
                                      inner sep     = 0pt,
                                      minimum width =4 pt}]

\def \p {0.5}

\foreach \x in {-10,-9,-8,-7,-6,-5,-4,-3,-2,-1,0,1,2,3,4,5,6,7,8,9,10}
\foreach \y in {0,1}
    \node at (\x,\y) {};

\foreach \x in {-10,-9,-8,-7,-6,-5,-4,-3,-2,-1,0,1,2,3,4,5,6,7,8,9}{
\foreach \y in {1,0}{
    \pgfmathparse{rnd}
    \let\dummynum=\pgfmathresult
    \ifdim\pgfmathresult pt < \p pt\relax \draw (\x,\y) -- (\x+1,\y);\fi
  }}

\foreach \x in {-10,-9,-8,-7,-6,-5,-4,-3,-2,-1,0,1,2,3,4,5,6,7,8,9,10}{
\foreach \y in {0}{
    \pgfmathparse{rnd}
    \let\dummynum=\pgfmathresult
    \ifdim\pgfmathresult pt < \p pt\relax \draw (\x,\y) -- (\x,\y+1);\fi
  }}

	\draw	(-9,1) -- (-8,1);
	\draw	(-7,0) -- (-6,0);
	\draw	(-1,0) -- (0,0);
	\draw	(3,0) -- (4,0);
	\draw	(8,0) -- (9,0);

	\node[draw=none,fill=none] at (-1,-0.75) {$R^{\mathrm{pre}}_{-1}$};
	\node[draw=none,fill=none] at (0,-0.75) {$\mathbf{0}$};
	\node[draw=none,fill=none] at (5,-0.75) {$R^{\mathrm{pre}}_{0}$};

	\draw[densely dotted] (-10.5,0) -- (-10,0);
	\draw[densely dotted] (-10.5,1) -- (-10,1);
	\draw[densely dotted] (10,0) -- (10.5,0);
	\draw[densely dotted] (10,1) -- (10.5,1);
\end{tikzpicture}\smallskip

\pgfmathsetseed{216642}
\begin{tikzpicture}[thin, scale=0.6,-,
                   shorten >=0pt+0.5*\pgflinewidth,
                   shorten <=0pt+0.5*\pgflinewidth,
                   every node/.style={circle,
                                      draw,
                                      fill          = black!80,
                                      inner sep     = 0pt,
                                      minimum width =4 pt}]

\def \p {0.5}

\foreach \x in {-10,-9,-8,-7,-6,-5,-4,-3,-2,-1,0,1,2,3,4,5,6,7,8,9,10}
\foreach \y in {0,1}
    \node at (\x,\y) {};

\foreach \x in {-10,-9,-8,-7,-6,-5,-4,-3,-2,-1,0,1,2,3,4,5,6,7,8,9}{
\foreach \y in {1,0}{
    \pgfmathparse{rnd}
    \let\dummynum=\pgfmathresult
    \ifdim\pgfmathresult pt < \p pt\relax \draw (\x,\y) -- (\x+1,\y);\fi
  }}

\foreach \x in {-10,-9,-8,-7,-6,-5,-4,-3,-2,-1,0,1,2,3,4,5,6,7,8,9,10}{
\foreach \y in {0}{
    \pgfmathparse{rnd}
    \let\dummynum=\pgfmathresult
    \ifdim\pgfmathresult pt < \p pt\relax \draw (\x,\y) -- (\x,\y+1);\fi
  }}

	\draw	(-9,1) -- (-8,1);
	\draw	(-7,0) -- (-6,0);
	\draw	(-1,0) -- (0,0);
	\draw	(3,0) -- (4,0);
	\draw	(8,0) -- (9,0);

	\node[draw=none,fill=none] at (-1,-0.75) {$R^{\mathrm{pre}}_{-1}$};
	\node[draw=none,fill=none] at (0,-0.75) {$\mathbf{0}$};
	\node[draw=none,fill=none] at (5,-0.75) {$R^{\mathrm{pre}}_{0}$};

	\draw[densely dotted] (-10.5,0) -- (-10,0);
	\draw[densely dotted] (-10.5,1) -- (-10,1);
	\draw[densely dotted] (10,0) -- (10.5,0);
	\draw[densely dotted] (10,1) -- (10.5,1);

	\draw[white]	(1,1) -- (4,1);
	\draw[white]	(6,1) -- (7,1);
	\node at (1,1) {};
	\node[white] at (-9,0) {};

\foreach \x in {-1,2,3,4,5,6,7}
\foreach \y in {1}
    \node[white] at (\x,\y) {};
\end{tikzpicture}
\end{center}
\caption{The original percolation configuration and the
backbone}	\label{Fig:backbone}
\end{figure}

Due to the Markovian structure of the percolation process under $\Prmp$,
there are infinitely many traps both to the left and to the right of the
origin $\bfnull$.
Let $T_n$, $n \in \Z$ be an enumeration of all trap pieces such that $T_n$ is
strictly to the left of $T_{n+1}$ for each $n \in \Z$
and that $T_1$ is the trap piece with minimal nonnegative $\x$-coordinate of the
trap entrance.
Denoting the length of the trap in the trap piece $T_n$ by $\ell_n$, the following result holds.

\begin{lemma}[Lemma 3.5 in \cite{Gantert+al:2018}]\label{Lem:trap length
distribution}
\begin{enumerate}[(a)] \itemsep0pt
\item	Under $\Prmp$, $(\ell_n)_{n \neq 0}$ is a family of i.i.d.\ positive
random variables independent of $\ell_0$
	with $\Prmp(\ell_1 = m) = (e^{2 \lambdacrit}-1) e^{-2 \lambdacrit m}$, $m
\in \N$.
\item	There is a constant $\chi(p)$ such that $\Prmp(\ell_0 = m) \leq
\chi(p) m e^{-2 \lambdacrit m}$, $m \in \N$.
\end{enumerate}
\end{lemma}

An excursion of the random walk $(Y_n)_{n \in \N_0}$ into a fixed trap of
length $m$ can be identified
with an excursion of a biased random walk $(S_n)_{n \in \N_0}$ on the line
graph $\{0,1,\ldots,m\}$ where $m$ is the length of the trap.
Therefore, we study biased random walk on $\{0,1,\ldots,m\}$.
Let
\begin{equation*}	\textstyle
\p \defeq \frac{e^\lambda}{e^\lambda+e^{-\lambda}},
\quad	\q \defeq 1-\p,
\quad	\text{and}	\quad
\gamma \defeq \frac{\q}{\p} = e^{-2 \lambda}.
\end{equation*}
We write $P^k_{m,\lambda}$
for the law of a biased random walk $(S_n)_{n \in \N_0}$ on
$\{0,\ldots,m\}$ starting at $k\in\{0,\ldots,m\}$,
moving to the right with probability $\p$ and moving left with probability
$\q$ from any vertex other than $0,m$.
The origin $0$ is supposed to be absorbing and at $m$ the walk stays put
with probability $\p$ and moves left with probability $\q$.
We write $E^k_{m,\lambda}$ for the corresponding expectation.
We drop the superscript $k$, both in $P^k_{m,\lambda}$ as well as
$E^k_{m,\lambda}$, if $k=1$.

For $k,l \in \{0,...,m\}$ we write $\sigma_k \defeq \inf\{j \in \N_0: S_j
= k\}$, $\sigma_k^{+} \defeq \inf\{j \in \N: S_j = k\}$,
and $\sigma_{k \to l} = \inf\{j \geq 0: S_j = l\}$ on $\{S_0=k\}$.
Let $\mathsf{e}_m \defeq P^m_{m,\lambda}(\sigma_0^{+} < \sigma_m^{+})$ be
the escape probability from the rightmost node in the trap to the trap
entrance
without a rebound to the rightmost node in the trap.
By the well-known Gambler's ruin formula, this is
\begin{equation}	\label{eq:Gambler's ruin}
\mathsf{e}_m	= P_{m,\lambda}^m(\sigma_0^{+} < \sigma_m^{+})
= \q \frac{\gamma^{m-1}-\gamma^{m}}{1-\gamma^{m}}
= \gamma^m \p \frac{1-\gamma}{1-\gamma^{m}}.
\end{equation}

\subsection{The proof of the lower bound}

We are ready to prove the lower bound.

\begin{lemma}	\label{Lem:tail estimate lower bound}
There exists some $c>0$ such that, for all $n\in\N$,
\begin{equation*}
\Prob(\tau_2-\tau_1\geq n ) \geq	c
n^{-\alpha}	\quad	\text{and}	\quad	\Prob(\tau_1\geq n) \geq c
n^{-\alpha}
\end{equation*}
\end{lemma}

In the next proof and throughout the paper, for a random variable $Z$ and
$\hat{p} \in (0,1)$, we write $Z \sim \text{geom}(\hat{p})$
if $Z$ is geometric with success parameter $\hat{p}$, i.e., $\Prob(Z=k) =
\hat{p} (1-\hat{p})^{k}$, $k\in\N_0$.

\begin{proof}
According to Lemma \ref{Lemma:known results rho and tau}, we find
\begin{equation*}
\Prob(\tau_2-\tau_1 \geq n) 
= \Prob^\circ(\tau_1 \geq n | Y_k \not = \bfnull \text{ for all } k \geq 1) 
\geq \Prob^\circ(\tau_1 \geq n, \, Y_k \not = \bfnull \text{ for all } k \geq 1).
\end{equation*}
On the other hand, as $\Prob(R^\mathrm{pre}_0 = \bfnull) >0$, we can
safely write
\begin{align*}
\Prob(\tau_1 \geq n) 
&\geq \Prob(R^\mathrm{pre}_0 = \bfnull) \Prob(\tau_1 \geq n,Y_k \not = \bfnull \text{ for all } k \geq 1 | R^\mathrm{pre}_0 = \bfnull). \\
&= 	\Prob(R^\mathrm{pre}_0 = \bfnull) \Prob^\circ(\tau_1 \geq n, Y_k \not = \bfnull \text{ for all } k \geq 1).
\end{align*}
We therefore provide a lower bound for $\Prob^\circ(\tau_1 \geq n,\, Y_k \not
= \bfnull \text{ for all } k \geq 1)$.
Under $\Prmp^\circ$, there is a pre-regeneration point at $\bfnull$ as
depicted in the figure below.
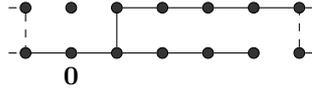
\vspace{-0.25cm}
\begin{center}
\begin{figure}[htbp]
\begin{tikzpicture}[thin, scale=0.6,-,
                   shorten >=2pt+0.5*\pgflinewidth,
                   shorten <=2pt+0.5*\pgflinewidth,
                   every node/.style={circle,
                                      draw,
                                      fill          = black!80,
                                      inner sep     = 0pt,
                                      minimum width =4 pt}]
\path[draw] 
	node at (0,0) {}
	node at (0,1) {}
	node at (1,0) {}
	node at (1,1) {}
	node at (2,0) {}
	node at (2,1) {}
	node at (3,0) {}
	node at (3,1) {}
	node at (4,0) {}
	node at (4,1) {}
	node at (5,0) {}
	node at (5,1) {}
	node at (-1,0){}
	node at (-1,1) {}
        ;

	\draw (1,0) -- (1,1);
	\draw (-1,0) -- (0,0);
	\draw (0,0) -- (1,0);
	\draw (1,1) -- (2,1);
	\draw (1,0) -- (2,0);
	\draw (2,0) -- (3,0);
	\draw (2,1) -- (3,1);
	\draw (3,0) -- (4,0);
	\draw (3,1) -- (4,1);
	\draw (4,1) -- (5,1);

\begin{scope}[dashed]  
	\draw (-1.5,0) -- (-1,0);
	\draw (-1.5,1) -- (-1,1);
	\draw (-1,1) -- (-1,0);

	\draw (5,0) -- (5,1);
	\draw (5,0) -- (5.5,0);
	\draw (5,1) -- (5.5,1);

\node[draw=none,fill=none] at (0,-0.5) {$\bfnull$};
\end{scope}
\end{tikzpicture}
\end{figure}
\end{center}
\vspace{-1cm}
Given there is a pre-regeneration point at $\bfnull$ (as is always the
case under $\Prmp^\circ$),
the law of the percolation cluster to the right of the origin under
$\Prmp$ and $\Prmp^\circ$ coincides
since the $\omega_n$, $n \in \N$ have the same law under $\Prmp$ and
$\Prmp^\circ$.
We may thus argue as on p.\;3404 of \cite{Axelson-Fisk+H"aggstr"om:2009b}
to conclude that the probability that
directly to the right of the origin, there is a trap of length $m$ as in
the picture above is
$\gamma(p) e^{-2\lambdacrit m}$ for some constant $\gamma(p) \in (0,1)$.

We write $T$ for the time spent on the first excursion of
$(Y_n)_{n\in\N_0}$ into the trap right of the origin. We have
\begin{equation*}
\Prob^\circ( \tau_1 \geq n, Y_k \not = \bfnull \text{ for all } k \geq 1 )
\geq \Prob^\circ( T \geq n, \text{there is a trap directly to the right of the origin}).
\end{equation*}
Typically, after entering the trap the walk drifts towards the bottom of
the trap and then requires a geometric number of trials to leave again.
It follows from the Gambler's ruin formula that for all $m$, hitting the bottom before leaving the trap
has positive probability bounded from below:
\begin{align*}
		P_{m,\lambda}^1 ( \sigma_m < \sigma_0 )
						=	\frac{ 1 - \gamma^1}{ 1 - \gamma^m}
						>	1 - \gamma
						>	0.
\end{align*}
The probability of leaving the trap from the bottom without rebound to the bottom is
$\mathsf{e}_m$. In order to visit the trap in the situation as depicted above, two
steps to the right at the start suffice. Thus we get
\begin{align*}
\Prob^\circ(\tau_1 \geq n, \, Y_k \not = \bfnull \text{ for all } k \geq 1)
&\geq \Big(\frac{e^\lambda}{ e^\lambda+1+e^{-\lambda} }\Big)^{\!2}
\sum_{m=2}^\infty \gamma(p)e^{-2\lambdacrit m} P_{m,\lambda}^1(T \geq n, \sigma_m < \sigma_0 ) \\
&\geq \frac{ (1-\gamma) e^{2\lambda} \gamma(p) }{ (e^\lambda+1+e^{-\lambda})^2 }
\sum_{m=2}^\infty e^{-2\lambdacrit m} (1-\mathsf{e}_m)^{n-1},
\end{align*}
Restricting this sum to the term of order $\hat{x} \defeq \frac{ \log n }{
|\log \gamma| }$ leads to
\begin{align*}
\Prob^\circ(\tau_1 \geq n, \, Y_k \not = \bfnull \text{ for all } k \geq 1)
&\geq \frac{ (1-\gamma) e^{2\lambda} \gamma(p) }{ (e^\lambda+1+e^{-\lambda})^2 } 
e^{-2\lambdacrit \lfloor \hat{x}\rfloor } 
(1 - \mathsf{e}_{\lfloor \hat{x} \rfloor} )^{n-1} \\
&\geq \frac{ (e^{2\lambda} -1) \gamma(p) }{ (e^\lambda+1+e^{-\lambda})^2 } e^{-2\lambdacrit \hat{x} } (1 - \mathsf{e}_{\hat{x}-1} )^{n-1}	\\
&= n^{-\alpha} \frac{ (e^{2\lambda} -1) \gamma(p) }{ (e^\lambda+1+e^{-\lambda})^2 } 
\exp( - \p (e^{2\lambda}-1)) (1+o_n(1)).
\end{align*}
\end{proof}

Let $(S'_n)_{n\in\N_0}$ be a biased random walk on $\Z$ that mimics the
steps of $(S_n)_{n\in\N_0}$ without staying put.
More precisely, set $S'_0 \defeq 0$ and for $n < \sigma_0$, let
\begin{align*}
& S'_{n+1} = S'_n + 1		\qquad	\text{if } S_{n+1} = S_n+1	\text{ or }
S_{n+1} = S_n = m,	\\
& S'_{n+1} = S'_n - 1		\qquad	\text{if } S_{n+1} = S_n-1.
\end{align*}
After $(S_n)_{n\in\N_0}$ hits the absorbing state 0, we let
$(S_n')_{n\in\N_0}$ move along as the usual biased random walk on $\Z$
with probability $\p$ to jump right.
For $z\in \Z$, write $P^z_{\Z,\lambda}$ and $E^z_{\Z,\lambda}$ for the law
of $(S'_n)_{n\in\N_0}$ starting at $S_0=z$ and the corresponding
expectation, respectively.
For $k\in\Z$, set
\begin{equation*}
\sigma_k^\Z		\defeq	\inf\{ l \geq 0: S'_l = k \}.
\end{equation*}
We start with a well-known fact about biased random walk on $\Z$.

\begin{lemma}	\label{Lem:moments of sigma_1^Z}	
For $x>0$, it holds that
\begin{equation*}
E^0_{\Z,\lambda} \big[ x^{\sigma_1^\Z} \big] = \frac{1 - \sqrt{ 1 - 4 \p \q x^2}}{2\q x}.
\end{equation*}
\end{lemma}
For completeness, we include a brief proof.
\begin{proof}
Let $x>0$ and $f(x) \defeq E^0_{\Z,\lambda} \big[ x^{\sigma_1^\Z} \big]$.
On the one hand, the Markov property gives
\begin{align}	\label{eqn:x^sigma_1^Z}
f(x) = \p x + \q x f(x)^2.
\end{align}
On the other hand, $\lim_{x \searrow 0} f(x) = 0$ due to dominated convergence.
Hence, solving \eqref{eqn:x^sigma_1^Z} for $f(x)$ yields the stated formula.
\end{proof}

We divide the time spent between the visits to the first and second
regeneration point $\tau_2-\tau_1$ as follows.
\begin{equation*}
\tau_2-\tau_1 = (\tau_2 - \tau_1)^\B + (\tau_2 - \tau_1)^\mathrm{traps}
\end{equation*}
where $(\tau_2-\tau_1)^\B$ and $(\tau_2 - \tau_1)^\mathrm{traps}$ are the
time spent in the backbone and in traps, respectively, during the time
interval $[\tau_1,\tau_2)$.
The following Lemma holds.

\begin{lemma}[Lemma 7.5 in \cite{Gantert+al:2018}]	\label{Lem:backbone
time power moments}
For any $\kappa>0$, we have $\E[((\tau_2 - \tau_1)^\B)^\kappa] < \infty$.
\end{lemma}

This and Markov's inequality imply the following result.

\begin{lemma}	\label{Lem:backbone time}
It holds that $\Prob((\tau_2 - \tau_1)^\B \geq n) = o(n^{-\alpha})$ as $n \to \infty$.
\end{lemma}

To obtain an upper bound on $\Prob(\tau_2 - \tau_1 \geq n)$, we thus need to
consider the time spent in traps.
We write $(\tau_2 - \tau_1)^\mathrm{traps}$ as
\begin{equation*}
(\tau_2 - \tau_1)^\mathrm{traps}	=	\sum_{i=1}^{T} \sum_{j=1}^{V_i} T_{ij},
\end{equation*}
where $T$ is the number of traps in $[\rho_1, \rho_2)$, $V_i$ is the
number of visits in the $i$th trap in $[\rho_1, \rho_2)$ and $T_{ij}$ is the time $(Y_n)_{n\in\N_0}$
spends during the $j$th excursion into the $i$th trap in $[\rho_1, \rho_2)$.

\subsection{Tail estimates for the time spent in a single trap}

If we fix a percolation environment $\omega$,
the time spent in a single trap of length $m$ can be split into the time
spent on
bottom-to-bottom excursions and the time spent to reach or leave the
bottom without a rebound to the left- or rightmost, respectively, node of
the trap.
This leads to the following result for a fixed number of excursions into a
single trap.

\begin{lemma}\label{Lem:quenched tail estimate for single trap}
Let $(S_{n,j})_{n \in \N_0}$, $j\in\N$ be i.i.d.\ copies of
$(S_n)_{n\in\N_0}$ starting at $1$.
Further, let $T_{ij}^\mathrm{qu,a}$ be the absorption time at 0 of the walk
$(S_{n,j})_{n\in\N_0}$, $j\in\N$.
Let $R \defeq E_{\Z,\lambda}^0[\sigma_1^\Z] = \frac1{1-2\q}$.
Then, for any $l \in \N$, there exist independent $Z_1,...,Z_l \sim
\text{geom}(\mathsf{e}_m)$ and $m_0 \in \N$ such that, for $m \geq m_0$ and $n \in \N$, we have
\begin{align*}
P_{m,\lambda}\bigg( \sum_{j=1}^l T_{ij}^\mathrm{qu,a} \geq n \bigg)
&\leq 2 P_{m,\lambda}\bigg( \sum_{j=1}^l Z_j \geq \frac{n}{4R} \bigg) \\
&\hphantom{\leq}
+ 3l \max \big\{
P_{m,\lambda}^1\big( \sigma_{1\to0} \geq \tfrac{n}{6l}, \sigma_0 < \sigma_m \big),	 P_{m,\lambda}^1\big( \sigma_{1\to m} \geq \tfrac{n}{6l}, \sigma_m < \sigma_0 \big), \\
&\hphantom{\leq + 3l \max \big\{}
P_{m,\lambda}^m\big( \sigma_{m\to0} \geq \tfrac{n}{6l}, \sigma_0 < \sigma_m^{+} \big)	\big\}.
\end{align*}
\end{lemma}
\begin{proof}
Let $Z^{(j)}$ be the number of returns to $m$ of $(S_{n,j})_{n\in\N_0}$
before absorption.
For completeness, we define $Z^{(j)} \defeq 0$ on the event where
$(S_{n,j})_{n\in\N_0}$ visits $m$ at most once.
By the strong Markov property, $P_{m,\lambda}(Z^{(j)} = k) =
P_{m,\lambda}^1(\sigma_m < \sigma_0) (1-\mathsf{e}_m)^{k} \mathsf{e}_m$ for $k \in \N$
and $P_{m,\lambda}(Z^{(j)}=0) = P_{m,\lambda}^1(\sigma_0<\sigma_m) +
P_{m,\lambda}^1(\sigma_m<\sigma_0)\mathsf{e}_m$.
We write $\tilde{T}_{jk}$, $k=1,\ldots, Z^{(j)}$ for the durations of
consecutive excursions of $(S_{n,j})_{n \in \N_0}$ from $m$ to $m$,
and let $\tilde{T}_{jk}$, $k> Z^{(j)}$, be a family of i.i.d.\ random
variables distributed as the duration of an excursion of $(S_n)_{n\in\N_0}$ from $m$ to
$m$ conditioned on the event $\{\sigma_m^+ < \sigma_0\}$.
When starting at 1, the walk $(S_n)_{n\in\N_0}$ either hits the absorbing
state 0 before reaching the trap bottom, or hits the bottom, does a
geometric number of bottom-to-bottom excursions, and then gets absorbed.
We have
\begin{align*}
P_{m,\lambda}\bigg( \sum_{j=1}^l T_{ij}^\mathrm{qu,a} \geq n \bigg)
&= P_{m,\lambda}\bigg( \sum_{j=1}^l T_{ij}^\mathrm{qu,a} \geq n, \bigg| 
\sum_{j=1}^l T_{ij}^\mathrm{qu,a} - \sum_{j=1}^l \sum_{k=1}^{Z^{(j)}}
\tilde{T}_{jk} \bigg| \leq \frac{n}{2} \bigg)	\\
&\hphantom{=} + P_{m,\lambda}\bigg( \sum_{j=1}^l T_{ij}^\mathrm{qu,a} \geq n, \bigg|
\sum_{j=1}^l T_{ij}^\mathrm{qu,a} - \sum_{j=1}^l \sum_{k=1}^{Z^{(j)}}
\tilde{T}_{jk} \bigg| > \frac{n}{2} \bigg)	\\
&\leq P_{m,\lambda}\bigg( \sum_{j=1}^l \sum_{k=1}^{Z^{(j)}} \tilde{T}_{jk} \geq \frac{n}{2} \bigg) \\
&\hphantom{\leq} + 3l \max\big\{
P_{m,\lambda}^1\big( \sigma_{1\to 0} \geq \tfrac{n}{6l}, \sigma_0 < \sigma_m\big),
P_{m,\lambda}^1\big( \sigma_{1\to m}\geq \tfrac{n}{6l}, \sigma_m < \sigma_0\big), \\
&\hphantom{\leq + 3l \max\big\{}
P_{m,\lambda}^m\big( \sigma_{m\to0} \geq \tfrac{n}{6l}, \sigma_0 < \sigma_m^{+} \big)	\big\}.
\end{align*}
We can safely replace $Z^{(j)}, j=1,...,l$
by an independent family of i.i.d.\ random variables $Z_j$ with law $\text{geom}(\mathsf{e}_m)$ under $P_{m,\lambda}$.
As $\tilde{T}_{jk}, j=1,...,l, k\in\N$ are nonnegative and i.i.d., we have
\begin{align*}
P_{m,\lambda}\bigg(\sum_{j=1}^l Z_j < n \bigg)
&= P_{m,\lambda}\bigg( \sum_{j=1}^l Z_j < n, \sum_{k=1}^n \tilde{T}_{1k} \geq 2Rn \bigg)
+ P_{m,\lambda}\bigg( \sum_{j=1}^l Z_j < n, \sum_{k=1}^n \tilde{T}_{1k} < 2Rn \bigg) \\
&\leq P_{m,\lambda}\bigg( \sum_{k=1}^n \tilde{T}_{1k} \geq 2Rn \bigg)
	+ P_{m,\lambda}\bigg( \sum_{k=1}^{Z_1+...+Z_l} \tilde{T}_{1k} < 2Rn \bigg) \\
&=	P_{m,\lambda}\bigg( \sum_{k=1}^n \tilde{T}_{1k} \geq 2Rn \bigg)
+ P_{m,\lambda}\bigg( \sum_{j=1}^l \sum_{k=1}^{Z_j} \tilde{T}_{jk} < 2Rn \bigg).
\end{align*}
This implies
\begin{align}\label{eqn:sum via trial number}
P_{m,\lambda}\bigg( \sum_{j=1}^l \sum_{k=1}^{Z_j} \tilde{T}_{jk} \geq 2Rn \bigg)
\leq P_{m,\lambda}\bigg( \sum_{j=1}^l Z_j \geq n \bigg)
+ P_{m,\lambda}\bigg( \sum_{k=1}^n \tilde{T}_{1k} \geq 2Rn \bigg).
\end{align}
Using Markov's inequality, the Markov property, stochastic domination
and Lemma \ref{Lem:moments of sigma_1^Z}, for $\mu>0$, we have
\begin{align*}
P_{m,\lambda}\bigg( \sum_{k=1}^n \tilde{T}_{1k} \geq 2Rn \bigg)
&\leq e^{-2\mu Rn} E_{m,\lambda}^m 
\big[ e^{ \mu \sigma_m^{+} } \big| \sigma_m^{+} < \sigma_0 \big]^n
\leq e^{-2\mu Rn} E_{\Z,\lambda}^0\big[ e^{\mu\sigma_1^\Z} \big]^n \\
&= e^{-2\mu Rn} \bigg( \frac{ 1 - \sqrt{ 1 - 4\p\q e^{2\mu} } }{ 2\q e^\mu } \bigg)^{\!n}.
\end{align*}
The function $f: \big[0,\frac{1}{2}\log\big(\frac{1}{4\p\q}\big) \big] \to \R$ given by
\begin{equation*}
f(\mu) 	\defeq	e^{-2\mu R} \frac{ 1 - \sqrt{ 1 - 4\p\q e^{2\mu} } }{ 2\q e^\mu }	
\end{equation*}
is differentiable and satisfies
\begin{equation*} 
f(0) = \frac{ 1 - (1-2\q) }{ 2\q } = 1, \qquad 
f'(0) = \frac{-1}{ 1 - 2\q } < 0.
\end{equation*}
Hence, there exists $\hat{\mu} > 0$ with $f(\hat{\mu}) < 1$, and
\begin{equation*}
P_{m,\lambda}\bigg( \sum_{k=1}^n \tilde{T}_{1k} \geq 2Rn \bigg)
\leq \bigg( \frac{ f(\hat{\mu}) }{ 1 - \mathsf{e}_m } \bigg)^{\!n} 
\cdot P_{m,\lambda}(Z_1 \geq n ).								
\end{equation*}
As $\mathsf{e}_m \to 0$ for $m\to\infty$, there exists $m_0$ such that $\frac{
f(\hat{\mu}) }{ 1-\mathsf{e}_m } < 1$ for all $m \geq m_0$. 
This and \eqref{eqn:sum via trial number} lead to
\begin{align*}
P_{m,\lambda}\bigg( \sum_{j=1}^l \sum_{k=1}^{Z_j} \tilde{T}_{jk} \geq 2Rn \bigg)
&\leq P_{m,\lambda}\bigg( \sum_{j=1}^l Z_j \geq n \bigg) 
+ \bigg(\frac{ f(\hat{\mu}) }{ 1-\mathsf{e}_{m_0} }\bigg)^{\!n} P_{m,\lambda}(Z_1 \geq n)  \\
&\leq 2 P_{m,\lambda}\bigg( \sum_{j=1}^l Z_j \geq n \bigg) 
\end{align*}
for $m \geq m_0$.
\end{proof}

Lemma \ref{Lem:quenched tail estimate for single trap} can be adapted to the case where the random walk is allowed to take lazy steps.
Let $(S_n^\mathrm{lazy})_{n\in\N_0}$ be the lazy biased random walk on the line graph $\{0,1,\ldots,m\}$ that moves to the right with probability $e^\lambda/(e^\lambda+1+e^{-\lambda})$,
to the left with probability $e^{-\lambda}/(e^\lambda+1+e^{-\lambda})$ and stays put with probability $1/(e^\lambda+1+e^{-\lambda})$ from any vertex other than $0,m$. 
The origin $0$ is again supposed to be absorbing and at $m$, the walk stays put with probability $(e^\lambda+1)/(e^\lambda+1+e^{-\lambda})$
and moves left with probability $e^{-\lambda}/(e^\lambda+1+e^{-\lambda})$.
Slightly abusing notation, we again write $P_{m,\lambda}$ for the law of $(S_n^\mathrm{lazy})_{n\in\N_0}$
starting at $S_0^\mathrm{lazy}=1$, and $E_{m,\lambda}$ for the corresponding expectation.

\begin{lemma}\label{Lem:quenched tail estimate for single trap - lazy}
Let $(S_{n,j}^\mathrm{lazy})_{n \in \N_0}$, $j\in\N$ be i.i.d.\ copies of $(S_n^\mathrm{lazy})_{n\in\N_0}$ starting at $1$.
Further, let $T_{ij}^\mathrm{qu}$ be the absorption time at 0 of the walk
$(S_{n,j}^\mathrm{lazy})_{n\in\N_0}$, $j\in\N$.
Let $R \defeq E_{\Z,\lambda}^0[\sigma_1^\Z] = \frac1{1-2\q}$ and $r_\lambda > e^{2\lambda}+e^\lambda$.
Then, for any $l \in \N$, there exist independent $Z_1,...,Z_l \sim \text{geom}(\mathsf{e}_m)$
and $m_1 \in \N$ such that, for $m \geq m_0\vee m_1$ and $n \in \N$, we have
\begin{align*}
P_{m,\lambda}\bigg(\sum_{j=1}^l T_{ij}^\mathrm{qu} \geq n\bigg)
&\leq 3 P_{m,\lambda}\bigg(\sum_{j=1}^l Z_j \geq \frac{n}{4r_\lambda R} \bigg) \\
&\hphantom{\leq}
+ 3l \max \big\{
P_{m,\lambda}^1\big( \sigma_{1\to0} \geq \tfrac{n}{6lr_\lambda}, \sigma_0 < \sigma_m \big),	 P_{m,\lambda}^1\big( \sigma_{1\to m} \geq \tfrac{n}{6lr_\lambda}, \sigma_m < \sigma_0 \big), \\
&\hphantom{\leq + 3l \max \big\{}
P_{m,\lambda}^m\big( \sigma_{m\to0} \geq \tfrac{n}{6lr_\lambda}, \sigma_0 < \sigma_m^{+} \big)	\big\}.
\end{align*}
\end{lemma}

\begin{proof}
We have
\begin{equation*}
\sum_{j=1}^l T_{ij}^\mathrm{qu} \disteq \sum_{j=1}^l \sum_{k=1}^{T_{ij}^\mathrm{qu,a}} \tilde{Z}_{k,j},
\end{equation*}
where $T_{ij}^\mathrm{qu,a}$, $j\in\N$ are as in Lemma \ref{Lem:quenched tail estimate for single trap}, 
and $\tilde{Z}_{k,j}$, $k,j\in\N$ are independent random variables distributed as the number of times the walk $(S_{n,j}^\mathrm{lazy})_{n\in\N_0}$ stays put before it changes its position for the $k$th time. 
Since the probability for $(S_{n,j}^\mathrm{lazy})_{n\in\N_0}$ to change its position at any vertex other than the absorbing state $0$ is bounded from below by $\tilde{p} \defeq e^{-\lambda}/(e^\lambda+1+e^{-\lambda})$, 
we have $\tilde{Z}_{k,j}\distleq Z_{k,j}$ where $Z_{k,j}$, $k,j\in\N$ is a family of i.i.d.\ geometric random variables with success probability $\tilde{p}$. 
Notice that $E_{m,\lambda}[Z_{1,1}] = (1-\tilde p)/\tilde p = e^{2\lambda}+e^\lambda > 2$. Choose $r_\lambda > e^{2\lambda}+e^\lambda$. Then,
as the $Z_{k,j}$, $k,j\in\N$ are nonnegative and i.i.d., we find
\begin{align*}
&P_{m,\lambda}\bigg(\sum_{j=1}^l T_{ij}^\mathrm{qu} \geq n\bigg)
\leq P_{m,\lambda}\bigg( \sum_{j=1}^l \sum_{k=1}^{T_{ij}^\mathrm{qu,a}} Z_{k,j} \geq n\bigg) \\
&~= P_{m,\lambda}\bigg(\!\sum_{j=1}^l \sum_{k=1}^{T_{ij}^\mathrm{qu,a}} Z_{k,j} \geq n, \sum_{j=1}^l T_{ij}^\mathrm{qu,a} > \Big\lfloor \frac{n}{r_\lambda} \Big\rfloor \bigg)
+ P_{m,\lambda}\bigg(\! \sum_{j=1}^l \sum_{k=1}^{T_{ij}^\mathrm{qu,a}} Z_{k,j} \geq n, \sum_{j=1}^l T_{ij}^\mathrm{qu,a} \leq \Big\lfloor \frac{n}{r_\lambda} \Big\rfloor\bigg) \\
&~\leq P_{m,\lambda}\bigg(\sum_{j=1}^l T_{ij}^\mathrm{qu,a} > \Big\lfloor \frac{n}{r_\lambda} \Big\rfloor \bigg) + P_{m,\lambda}\bigg(\sum_{k=1}^{\lfloor\frac{n}{r_\lambda}\rfloor} Z_{k,1} \geq n \bigg).
\end{align*}
Standard large deviation estimates yield that $P_{m,\lambda}(\sum_{k=1}^{\lfloor n/r_\lambda\rfloor} Z_{k,1} \geq n)$ decays exponentially fast as $n \to \infty$ (with a rate which is independent of $m$). 
Hence, as $\mathsf{e}_m\to0$ for $m\to\infty$, there exists $m_1 = m_1(\lambda) \in\N$ such that for all $m\geq m_1$
\begin{align*}
P_{m,\lambda}\Bigg(\sum_{k=1}^{\lfloor\frac{n}{r_\lambda}\rfloor} Z_{k,1} \geq n \Bigg)
&\leq (1-\mathsf{e}_m)^{\lceil \frac{n}{4 r_\lambda R} \rceil} = P_{m,\lambda}\Big(Z_1 \geq \frac{n}{4 r_\lambda R}\Big).
\end{align*}
The remainder of the proof now follows from Lemma \ref{Lem:quenched tail estimate for single trap}.
\end{proof}

In the annealed case, Lemma \ref{Lem:quenched tail estimate for single trap - lazy} 
translates into a tail probability of basically order $n^{-\alpha}$
(given the trap is actually seen).

\begin{lemma}\label{Lem:annealed tail estimate for single trap}
Let $R,r_\lambda,m_0,m_1$ be as in Lemma \ref{Lem:quenched tail estimate for single trap - lazy}
and $\mu>0$ be such that $E_{\Z,\lambda}^0\big[ e^{\mu\sigma_1^\Z} \big] < \infty$.
Further, let $T^{\mathrm{ann}}_{ij}$, $i\in\Z$, $j\in\N$ be a family of random variables which are independent given $\omega$
and with $T^{\mathrm{ann}}_{ij}$ given $\omega$ being distributed as the hitting time of the entrance of the trap in $T_i$ by $(Y_n)_{n\in\N_0}$ under $P_{\omega,\lambda}$
when $(Y_n)_{n\in\N_0}$ starts at the right neighbor of the trap entrance.
Then
\begin{align*}
\Prob\bigg( \sum_{j=1}^l T^{\mathrm{ann}}_{ij} \geq n, \ell_i \geq m_0 \vee m_1 \bigg)
\leq
\begin{cases}
c_1 l^{\alpha+1} n^{-\alpha} + c_2 l e^{ - \mu \tfrac{n}{6lr_\lambda}},			&	\text{for }	i\not = 0,	\\
c_1' l^{\alpha+1} n^{-\alpha} \log n + c_2' l e^{ - \mu \tfrac{n}{6lr_\lambda} }	&	\text{for }	i = 0,
\end{cases}
\end{align*}
where $c_1 = c_1(p,\lambda), c_2 = c_2(p,\lambda), c_1' = c_1'(p,\lambda), c_2' = c_2'(p,\lambda)$ are positive, finite constants neither depending on $n$ nor $l$.
\end{lemma}
\begin{proof}
Using Lemmas \ref{Lem:trap length distribution} and \ref{Lem:quenched tail estimate for single trap - lazy},
we can estimate $\Prob\big( \sum_{j=1}^l T^{\mathrm{ann}}_{ij} \geq n, \ell_i \geq m_0 \vee m_1 \big)$
using independent $Z_1,...,Z_l \sim \text{geom}(\mathsf{e}_m)$ and $T_{ij}^\mathrm{qu}$,
$j=1,\ldots,l$, $r_\lambda$ and $R$ as defined in Lemma \ref{Lem:quenched tail estimate for single trap - lazy} by
\begin{align}\notag
\Prob&\bigg( \sum_{j=1}^l T^{\mathrm{ann}}_{ij} \geq n, \ell_i \geq m_0 \vee m_1 \bigg)
=\sum_{m=m_0\vee m_1}^\infty \Prmp( \ell_i = m ) P_{m,\lambda}\bigg(
\sum_{j=1}^l T_{ij}^\mathrm{qu} \geq n \bigg) \\\notag
&\leq 3 \sum_{m=m_0\vee m_1}^\infty \alpha_i(m) e^{-2\lambdacrit m} 
P_{m,\lambda}\bigg( \sum_{j=1}^l Z_j \geq \frac{n}{4r_\lambda R} \bigg) \\\notag
&\hphantom{\leq} + 3l \!\!\! \sum_{m=m_0\vee m_1}^\infty \!\!\! \alpha_i(m) e^{-2\lambdacrit m} \max\big\{ 
P_{m,\lambda}^1\big( \sigma_{1\to0} \geq \tfrac{n}{6lr_\lambda}, \sigma_0 < \sigma_m \big), P_{m,\lambda}^1\big( \sigma_{1\to m} \geq \tfrac{n}{6lr_\lambda}, \sigma_m < \sigma_0 \big), \\
&\hphantom{\leq + 3l \!\!\! \sum_{m=m_0\vee m_1}^\infty \!\!\! \alpha_i(m) e^{-2\lambdacrit m} \max\big\{}
P_{m,\lambda}^m\big( \sigma_{m\to0} \geq \tfrac{n}{6lr_\lambda}, \sigma_0 < \sigma_m^{+} \big)
\big\},		\label{eqn:application of quenched estimate}
\end{align}
where $\alpha_i(m) \defeq (e^{2\lambdacrit}-1)$ for $i\not = 0$ and $\alpha_0(m) \defeq \chi(p)m$.
We consider the second series first. For $y \in \{0,...,m\}$ we write
\begin{equation*} 
		h(y)	\defeq	P_{m,\lambda}^y( \sigma_0 < \sigma_m ).				
\end{equation*}
Due to the Gambler's ruin formula we have $h(y) = \frac{ \gamma^y -
\gamma^m }{ 1 - \gamma^m }$.
An excursion of $(S_n)_{n\in\N_0}$ starting from either 1 or $m$ to the origin 0 conditioned on $\sigma_0 < \sigma_m^+$ has the transition probabilities
\begin{equation*} 
P_{m,\lambda}^y( S_1 = z | \sigma_0 < \sigma_m^+ ) = \frac{h(z)}{h(y)} p(y,z),		
\end{equation*}
where $y\in \{1,...,m-1\}$, $z\in\{0,...,m\}$ and $p(y,z) \defeq
P_{m,\lambda}^y(S_1 = z)$. For $y\in\{1,...,m-1\}$ this implies
\begin{equation*} 
\frac{ P_{m,\lambda}^y( S_1 = y + 1 | \sigma_0 < \sigma_m^+ ) }
{ P_{m,\lambda}^y( S_1 = y - 1 | \sigma_0 < \sigma_m^+ ) }
= \frac{ h(y+1) }{ h(y-1) } \frac{p(y,y+1) }{p(y,y-1) }
< \gamma,
\end{equation*}
whereas
\begin{equation*} 
\frac{ P_{m,\lambda}^m( S_1 = m | \sigma_0 < \sigma_m^+ ) }{ P_{m,\lambda}^m( S_1 = m-1 | \sigma_0 < \sigma_m^+ ) }
= 0 
< \gamma.	
\end{equation*}
In other words, conditioned on $\sigma_0 < \sigma_m^+$, the walk
$(S_n)_{n\in\N_0}$ drifts towards to the left at least as strong as the unconditioned walk drifts towards the right. Estimating all three quantities in the
$\max$-term by corresponding quantities for $(S'_n)_{n\in\N_0}$, the
biased random walk on $\Z$, we get
\begin{align*}
\max& \big\{ 
P_{m,\lambda}^1\big( \sigma_{1\to0} \geq \tfrac{n}{6lr_\lambda}, \sigma_0 < \sigma_m \big),
P_{m,\lambda}^1\big( \sigma_{1\to m} \geq \tfrac{n}{6lr_\lambda}, \sigma_m < \sigma_0 \big),
P_{m,\lambda}^m\big( \sigma_{m\to0} \geq \tfrac{n}{6lr_\lambda}, \sigma_0 < \sigma_m^{+} \big)
\big\} \\
&\leq \max\big\{ 
P_{\Z,\lambda}^0 \big( \sigma_1^\Z \geq \tfrac{n}{6lr_\lambda} \big),
P_{\Z,\lambda}^1\big( \sigma_{m}^\Z \geq \tfrac{n}{6lr_\lambda} \big),
P_{\Z,\lambda}^0\big( \sigma_{m}^\Z \geq \tfrac{n}{6lr_\lambda} \big)
\big\} \\
&=	P_{\Z,\lambda}^0\big( \sigma_{m}^\Z \geq \tfrac{n}{6lr_\lambda} \big).
\end{align*}
Using Markov's inequality and Lemma \ref{Lem:moments of sigma_1^Z}, we get
that for $\mu>0$ with $E_{\Z,\lambda}^0\big[e^{\mu \sigma_1^\Z} \big] <
\infty$,
\begin{align*}
3l  \!\!\!\sum_{m=m_0\vee m_1}^\infty & \!\!\! \alpha_i(m)
e^{-2\lambdacrit m} P_{\Z,\lambda}^0\big( \sigma_{m}^\Z \geq \tfrac{n}{6lr_\lambda} \big) \\
&\leq 3 l e^{ -\mu\frac{n}{6lr_\lambda}} \!\!\! \sum_{m=m_0\vee m_1}^\infty \!\!\! \alpha_i(m) e^{-2\lambdacrit m} E_{\Z,\lambda}^0\big[e^{\mu\sigma_1^\Z}\big]^{m}	\\
&= 3 l e^{ -\mu\frac{n}{6lr_\lambda}} \!\!\! \sum_{m=m_0\vee m_1}^\infty \!\!\! \alpha_i(m) e^{-2\lambdacrit m} \bigg( \frac{ 1-\sqrt{ 1-4\p\q e^{2\mu} } }{ 2\q e^{\mu} } \bigg)^{\!m}.
\end{align*}
The latter series is finite. To see this, notice that if $\lambda < \lambdacrit$, we have $e^{-2\lambdacrit} < e^{-2\lambda} = \frac{\q}{\p}$
and thus
\begin{equation*} 
e^{-2\lambdacrit} \frac{ 1-\sqrt{ 1-4\p\q e^{2\mu} } }{ 2\q e^{\mu} } < 1
\end{equation*}
as $1 - \sqrt{1 - 4\p\q e^{2\mu}} \leq 1$ and $2\p e^{\mu} >1$. If on the
other hand $\lambda \geq \lambdacrit$,
we have $E_{\Z,\lambda}^0[e^{\mu\sigma_1^\Z}] \leq
E_{\Z,\lambdacrit}^0[e^{\mu\sigma_1^\Z}]$
and the series converges using the same argument.

\noindent
For the first series on the right-hand side of \eqref{eqn:application of quenched estimate}, we use the union bound to get 
\begin{align*}
3 \!\! \sum_{m=m_0\vee m_1}^\infty \!\! \alpha_i(m) e^{-2\lambdacrit m} 
P_{m,\lambda}\bigg( \sum_{j=1}^l Z_j \geq \frac{n}{4r_\lambda R} \bigg)
&\leq 3l \!\! \sum_{m=m_0\vee m_1}^\infty \!\! \alpha_i(m) e^{-2\lambdacrit m} P_{m,\lambda}\bigg( Z_1 \geq \frac{n}{4r_\lambda Rl} \bigg) \\
&= 3l \!\! \sum_{m=m_0\vee m_1}^\infty \!\! \alpha_i(m) e^{-2\lambdacrit m} (1-\mathsf{e}_m)^{\lceil \frac{n}{4r_\lambda Rl}\rceil}.
\end{align*}
We set $n_0 \defeq \lceil\frac{n}{4r_\lambda Rl}\rceil$. Since $\mathsf{e}_m \geq (\p-\q)\gamma^m$, we have 
\begin{equation*}
3l \sum_{m=m_0\vee m_1}^\infty \!\! \alpha_i(m) e^{-2\lambdacrit m} (1-\mathsf{e}_m)^{n_0} 
\leq 3l \sum_{m=m_0\vee m_1}^\infty \!\! \alpha_i(m) e^{-2\lambdacrit m} (1-(\p-\q)\gamma^m)^{n_0}.
\end{equation*} 
Let $t\in\N_0$ be such that $(\p-\q)\gamma^{-t} \leq 1< (\p-\q)\gamma^{-(t+1)}$. Then 
\begin{align*}
3l \sum_{m=m_0\vee m_1}^\infty \alpha_i(m) &e^{-2\lambdacrit m} (1-(\p-\q)\gamma^m)^{n_0} \\
&\leq 3l e^{2\lambdacrit t} \sum_{m=m_0\vee m_1}^\infty \alpha_i(m+t) e^{-2\lambdacrit(m+t)} \Big(1-\frac{\p-\q}{\gamma^t}\gamma^{m+t}\Big)^{n_0} \\
&\leq 3l e^{2\lambdacrit t} \sum_{m=0}^\infty \alpha_i(m) e^{-2\lambdacrit m} \Big(1-\frac{\p-\q}{\gamma^t}\gamma^m\Big)^{n_0} \\
&= 3l e^{2\lambdacrit t} \sum_{m=0}^\infty \alpha_i(m) e^{-2\lambdacrit m} \sum_{j=0}^{n_0} \binom{n_0}{j} \Big(-\frac{\p-\q}{\gamma^t}\gamma^m\Big)^j 1^{n_0-j} \\
&= 3l e^{2\lambdacrit t} \sum_{j=0}^{n_0} \binom{n_0}{j}(-1)^j \Big(\frac{\p-\q}{\gamma^t}\Big)^j \sum_{m=0}^{\infty} \alpha_i(m) \gamma^{\alpha m+jm} \\
&= \begin{cases} 3l e^{2\lambdacrit t} (e^{2\lambdacrit}-1) \sum_{j=0}^{n_0} \binom{n_0}{j} (-1)^j \big(\frac{\p-\q}{\gamma^t}\big)^j \frac{1}{1-\gamma^{\alpha+j}} &\text{if $i\neq 0$},\\
3l e^{2\lambdacrit t} \chi(p) \sum_{j=0}^{n_0} \binom{n_0}{j} (-1)^j \big(\frac{\p-\q}{\gamma^t}\big)^j \frac{\gamma^{\alpha+j}}{(1-\gamma^{\alpha+j})^2}&\text{if $i=0$}.
\end{cases}
\end{align*}
To find the asymptotic behavior of the two expressions in the Lemma, we apply residue calculus.
Define the complex function $\phi$ via $\phi(z) \defeq \frac{(\p-\q)^z}{\gamma^{tz}(1-\gamma^{\alpha+z})}$ for $z \in \C$.
Then $\phi$ is holomorphic in $\C$ except at the poles $z_k \defeq \frac{2 \pi \imag k}{\log\gamma}-\alpha$, $k\in\Z$.
Moreover, by the choice of $t$, $|\phi(z)|$ remains bounded as $|\Real(z)| \to \infty$. Consequently, Theorem 2(i) in \cite{Flajolet+Sedgewick:1995} applies and gives
\begin{equation*}
\sum_{j=0}^{n_0} \binom{n_0}{j} (-1)^j \Big(\frac{\p-\q}{\gamma^{t}}\Big)^j \frac{1}{1-\gamma^{\alpha+j}}
= (-1)^{n_0+1} \sum_{k\in\Z} \underset{z=z_k}{\text{Res}} \Big( \frac{1}{1-\gamma^{\alpha+z}} \frac{(\p-\q)^z n_0! }{\gamma^{tz}z(z-1)\ldots(z-n_0)} \Big).
\end{equation*}
Along the lines of Example 3 in \cite{Flajolet+Sedgewick:1995}, we get 
\begin{align}\notag
\sum_{j=0}^{n_0} \binom{n_0}{j} (-1)^j&\Big(\frac{\p-\q}{\gamma^{t}}\Big)^j \frac{1}{1-\gamma^{\alpha+j}}
= \Big( -\frac{1}{\log\gamma} \Big) \sum_{k\in\Z} \frac{\Gamma(n_0+1)}{\Gamma(n_0+1-z_k)} \Gamma(-z_k) \Big(\frac{\p-\q}{\gamma^{t}}\Big)^{\!z_k} \\\notag
&= \frac{1}{2\lambda} n_0^{-\alpha} \sum_{k\in\Z} n_0^{\frac{2 \pi \imag k}{\log\gamma}} \frac{\Gamma(n_0+1)n_0^{-z_k}}{\Gamma(n_0+1-z_k)} \Gamma(-z_k) \Big(\frac{\p-\q}{\gamma^{t}}\Big)^{z_k} \\\notag
&\leq \frac{\gamma^{t\alpha}(4r_\lambda R)^\alpha l^\alpha}{2\lambda(\p-\q)^\alpha} n^{-\alpha} \sum_{k\in\Z} e^{2 \pi \imag k (\log_\gamma((\p-\q)n_0)-t)} \frac{\Gamma(n_0+1)n_0^{-z_k}}{\Gamma(n_0+1-z_k)} \Gamma(-z_k) \\
&= \frac{e^{-2\lambdacrit t}(4r_\lambda R)^\alpha l^\alpha}{2\lambda(\p-\q)^\alpha} n^{-\alpha} \sum_{k\in\Z} e^{2 \pi \imag k \log_\gamma((\p-\q)n_0)} \frac{\Gamma(n_0+1)n_0^{-z_k}}{\Gamma(n_0+1-z_k)} \Gamma(-z_k)\label{eqn:residue_ineq0}
\end{align}
where $\Gamma$ is the complex gamma function.
From Stirling's formula, e.\,g.\ \cite[Theorem 1.4.2]{Andrews+Askey+Roy:1999}, we know that
\begin{equation*}
\log \Gamma(z) = \frac12 \log(2\pi) + \Big(z-\frac12\Big) \log z - z + R(z)
\end{equation*}
for $z\in\C \setminus (-\infty,0]$ where $\log$ is the branch of the complex logarithm, defined on $\C \setminus (-\infty,0]$,
with $\log x \in \R$ for all $x > 0$ and where $R(z)$ satisfies $|R(z)| \leq \frac{c}{|z|}$ for some constant $c>0$.
Hence,
\begin{align*}
\frac{\Gamma(n_0\!+\!1)n_0^{-z_k}}{\Gamma(n_0\!+\!1\!-\!z_k)}
&=
\exp\bigg(\big(n_0\!+\!\tfrac12\big) \log (n_0\!+\!1) - (n_0\!+\!1) + R(n_0\!+\!1)	 - z_k \log n_0 \\
&\hphantom{=\exp\bigg(}
- \Big(\big(n_0\!+\!\tfrac12\!-\!z_k\big) \log (n_0\!+\!1\!-\!z_k) - (n_0\!+\!1\!-\!z_k) + R(n_0\!+\!1\!-\!z_k)\Big) \bigg)	\\
&=
\Big(\frac{n_0\!+\!1\!-\!z_k}{n_0}\Big)^{\!z_k}
\Big(\frac{n_0\!+\!1}{n_0\!+\!1\!-\!z_k}\Big)^{\!n_0\!+\!\tfrac12}
e^{-z_k} e^{R(n_0\!+\!1)-R(n_0\!+\!1\!-\!z_k)}.
\end{align*}
In this product, the first and second factors are bounded in absolute value by $1$, the third by $e^\alpha$, and the fourth by $e^{2c}$.
Using Corollary 1.4.4 in \cite{Andrews+Askey+Roy:1999}, we conclude that $|\Gamma(-z_k)| \to 0$ exponentially fast as $|k| \to \infty$
and that the bi-infinite series in \eqref{eqn:residue_ineq0} is finite and can be bounded by a finite constant $c_1$  that neither depends on $n$ nor $l$.

\noindent
For $i=0$, we again use Theorem 2(i) in \cite{Flajolet+Sedgewick:1995} and find 
\begin{align*}
\sum_{j=0}^{n_0} \! \binom{n_0}{j} (-1)^j \Big(\frac{\p\!-\!\q}{\gamma^{t}}\Big)^j \frac{\gamma^{\alpha+j}}{(1\!-\!\gamma^{\alpha+j})^2}
= (-1)^{n_0+1} \sum_{k\in\Z} \underset{z=z_k}{\text{Res}} \Big(\frac{\gamma^{\alpha+z}}{(1\!-\!\gamma^{\alpha+z})^2} \frac{(\p\!-\!\q)^z n_0!}{\gamma^{tz}z(z\!-\!1)\ldots(z\!-\!n_0)} \Big)
\end{align*}
with $z_k = \frac{2 \pi \imag k}{\log\gamma}-\alpha$ as above. Evaluating the residues leads to
\begin{align*}
&\sum_{j=0}^{n_0} \binom{n_0}{j} (-1)^j \Big(\frac{\p\!-\!\q}{\gamma^{t}}\Big)^j \frac{\gamma^{\alpha+j}}{(1-\gamma^{\alpha+j})^2} \\
&= (-1)^{n_0+1} \sum_{k\in\Z} \frac{n_0!}{(\log\gamma)^2 z_k(z_k-1)\ldots(z_k-n_0)} \Big( \log\!\Big(\frac{\p\!-\!\q}{\gamma^t}\Big) + \sum_{j=0}^{n_0} \frac{1}{j\!-\!z_k} \Big)\Big(\frac{\p\!-\!\q}{\gamma^{t}}\Big)^{z_k} \\
&= \frac{e^{-2\lambdacrit t}}{4\lambda^2(\p\!-\!\q)^\alpha} n_0^{-\alpha} \sum_{k\in\Z} e^{2 \pi \imag k\log_\gamma((\p-\q)n_0)}
\frac{\Gamma(n_0\!+\!1)n_0^{-z_k}}{\Gamma(n_0\!+\!1\!-\!z_k)} \Gamma(-z_k) \Big( \log\!\Big(\frac{\p\!-\!\q}{\gamma^t}\Big) +\sum_{j=0}^{n_0} \frac{1}{j\!-\!z_k} \Big) \\
&\leq \frac{e^{-2\lambdacrit t}(4r_\lambda R)^\alpha}{4\lambda^2(\p\!-\!\q)^\alpha} l^\alpha n^{-\alpha} \sum_{k\in\Z} e^{2 \pi \imag k\log_\gamma((\p-\q)n_0)}
\frac{\Gamma(n_0\!+\!1)n_0^{-z_k}}{\Gamma(n_0\!+\!1\!-\!z_k)} \Gamma(-z_k) \sum_{j=0}^{n_0} \frac{1}{j\!-\!z_k}.
\end{align*}
Along the same lines as above, we can show that this bi-infinite series has finite value and the whole term can be bounded by $c_1' l^\alpha n^{-\alpha}\log n$, where $c_1' \in (0,\infty)$ does not depend on $n$ or $l$.
\end{proof}

\subsection{A coupling} \label{section:coupling}

As the times spent in different traps
are not independent, further work is needed to transfer the tail
estimate for the time
spent in a single trap to the time spent in the possibly several traps
inside a block $[\rho_i, \rho_{i+1})$.
Therefore,
we introduce a random walk on a subgraph $\omegaprune$ of the initial
environment $\omega$ as follows.
We take the initial graph $\omega$ sampled according to $\Prmp$ or
$\Prmp^\circ$	
and modify it as follows. For each trap
$P = (e_1,\ldots,e_m)$ in $\omega$
with trap entrance $u_0$
and edges $e_1 = \< u_0,u_1 \>,\ldots,e_m = \<u_{m-1},u_m\>$,
we delete the edges $e_1,\ldots,e_m$ from $\omega$ and also the vertices
$u_1,\ldots,u_m$.
We further delete the opposite vertices $u_1',\ldots,u_m'$
and replace the parallel edges $e_1',\ldots,e_m',\<u_m',u_m'+(1,0)\>$
with a single edge connecting $u_0'$ and $u_m'+(1,0)$ with resistance
given by the sum of the resistances of the single edges.
We shall call the vertex $u_0'$ opposite the former trap entrance an
\emph{obstacle}. Should this procedure lead to the deletion of $\bfnull$,
we assign $\x$-coordinate $0$ in $\omegaprune$ to the obstacle that
replaced the trap piece which contained $\bfnull$ in $\omega$.
In this way, we also obtain new conductances $c^\mathrm{s}$ on $\omegaprune$.

\begin{figure}[H] 
\begin{center}
\begin{tikzpicture}[thin, scale=0.6,-,
                   shorten >=0pt+0.5*\pgflinewidth,
                   shorten <=0pt+0.5*\pgflinewidth,
                   every node/.style={circle, draw, fill = black, inner
sep = 0pt, minimum width =4 pt}
                   ]

\foreach \x in {20,21,22,23,24,25,26}{
\foreach \y in {0,1}{
    \node at (\x,\y) {};
}}
\node[rectangle, draw, fill = black, minimum width = 4pt, minimum
height=4pt] at (22,1) {};
\node[rectangle, draw, fill = black, minimum width = 4pt, minimum
height=4pt] at (25,1) {};
\foreach \x in {20,21,22,24,25}{
\foreach \y in {1}{
    \ifdim \x pt < -2 pt \draw (\x,\y) -- (\x+1,\y);\fi
    \ifdim \x pt > -2 pt \draw (\x,\y) -- (\x+1,\y);\fi
  }}
\foreach \x in {21,23}{
	\draw (\x,0) -- (\x+1,0);
	}
\foreach \x in {21,22,23,24,25}{
	\draw (\x,0) -- (\x,1);
	}
\draw[densely dotted] (19.5,1) -- (20,1);
\draw[densely dotted] (26,1) -- (26.5,1);

\foreach \x in {5,6,7,8,9,10,11,12,13,14,15}
\foreach \y in {0,1}
    \node at (\x,\y) {};
\foreach \x in {5,6,7,8,10,11,12,13,14}{
\foreach \y in {1}{
    \ifdim \x pt < -1 pt \draw (\x,\y) -- (\x+1,\y);\fi
    \ifdim \x pt > -1 pt \draw (\x,\y) -- (\x+1,\y);\fi
  }}
\foreach \x in {6,7,9,11,12,13}{
\foreach \y in {0}{
	\draw (\x,\y) -- (\x+1,\y);
 }}
\foreach \x in {6,7,9,10,11}{
	\draw (\x,0) -- (\x,1);
	}
\draw[densely dotted] (4.5,1) -- (5,1);
\draw[densely dotted] (15,1) -- (15.5,1);

\node[draw=none,fill=none] at (6,-0.5) {$\mathbf{0}$};
\node[draw=none,fill=none] at (9,1.5) {$(3,\!1)$};
\node[draw=none,fill=none] at (21,-0.5) {$\mathbf{0}$};
\node[draw=none,fill=none] at (23,1.5) {$(2,\!1)$};
\end{tikzpicture}
\end{center}
\caption{Comparison of $\omega$ (left) and the resulting $\omegaprune$
(right). Normal vertices are drawn as filled circles, the obstacles as
filled boxes.}
\end{figure}
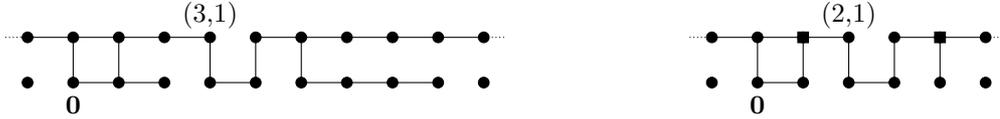
By the series law, the corresponding resistances $r^\mathrm{s}$ between
the first obstacle $v$ to the right of $\bfnull$
that replaces a trap piece covering $\x$-level $k$ to $k+m+1$ and its
neighbors $u$ to the left and $w$ to the right satisfy
\begin{equation*}
r^\mathrm{s}( \<u,v\> ) = r( \<u,v\> ) = e^{-\lambda(\x(u) + \x(v))} =
e^{-\lambda(2k-1)}
\end{equation*}
and
\begin{align*}
r^\mathrm{s}( \<v,w\> )
&= \sum_{j=k}^{k+m} r( \<j,\y(v)\> , \<j+1,\y(v)\> ) = \sum_{j=k}^{k+m}
e^{-\lambda(2j+1)}
= e^{-\lambda(2k+1)} \frac{1 - e^{-2\lambda(m+1)} }{ 1 - e^{-2\lambda } }.
\end{align*}
Based on this, we define the \emph{pruned random walk} as the lazy random
walk $(Y_n^\mathrm{p})_{n\in\N_0}$ on $\omegaprune$
with transition probabilities proportional to the conductances
\begin{equation*}
\cp( \<u,v\> )	=	e^{\lambda ( \x(u) + \x(v) ) } \cdot (1 -
e^{-2\lambda})^{p(v)}
\end{equation*}
where $\x(u) \leq \x(v)$ and $p(v)$ is the number of obstacles with
$\x$-coordinate $\in [0,\x(v))$.
More precisely, if $Y_n^\mathrm{p}=u$, then the walk attempts to step from
$u$ to $v$ with probability proportional to $\cp( \<u,v\> )$.
If the edge between $u$ and $v$ is present in $\omegaprune$, then the step
is actually performed,
otherwise the walk stays put.

Roughly speaking, $(Y_n^\mathrm{p})_{n\in\N_0}$ is the lazy random walk on
the non-trap pieces of $\omega$ when all traps are set to have infinite
length.
Intuitively, as the traps in $\omega$ have finite lengths,
the embedding of $(Y_n^\mathrm{p})_{n\in\N_0}$ into $\omega$ will lag
behind the random walk $(Y_n)_{n\in\N_0}$.
Regenerations of $(Y_n^\mathrm{p})_{n\in\N_0}$ also amount to
regenerations of $(Y_n)_{n\in\N_0}$
without implications on the lengths of the traps in the underlying piece
of $\omega$.
Furthermore, $(Y_n^\mathrm{p})_{n\in\N_0}$ can be used to bound the number
of visits to any trap by a quantity independent of the trap lengths,
thus greatly reducing the difficulties in transforming the estimate of Lemma
\ref{Lem:annealed tail estimate for single trap}
to an estimate for the time spent in the whole block $[\rho_i,\rho_{i+1})$ in $\omega$.
To make this precise, we give a coupling of $(Y_n^\mathrm{p})_{n\in\N_0}$
and $(Y_n)_{n\in\N_0}$ with the described properties.
Technically, the coupling is such that we obtain processes with the same
distributions as $(Y_n)_{n\in\N_0}$ and $(Y_n^\mathrm{p})_{n\in\N_0}$ and
the desired properties,
but we shall again refer to them as $(Y_n)_{n\in\N_0}$ and
$(Y_n^\mathrm{p})_{n\in\N_0}$, respectively, once equality of the
corresponding laws is established.

First, let $(O_i)_{i\in\Z}$ be an enumeration of the obstacles in
$\omegaprune$ such that $\ldots < \x(O_{-1}) < 0 \leq \x(O_0) < \x(O_2) <
\ldots$.
Starting from $\omegaprune$, take an independent family $(L_i)_{i\in\Z}$
of random variables, with $(L_i)_{i \not = 0}$ independent of $\omega$.
We re-insert at $O_i$ a trap piece with a trap of length $L_i$.
Here, we let $L_i$ have the same distribution as $\ell_i$ for $i\not=0$.
For $i=0$, let
the law of $L_0$ given $\x(O_0) > 0$ be the law of $\ell_1$.
Further notice that if $\x(O_0) = 0$, then, by the definition of $T_0$ and $T_1$,
either $\bfnull$ is one of the two leftmost vertices in $T_1$ or $\bfnull \in \inner(T_0)$ which consists of all vertices from $T_0$
except the two leftmost and the two rightmost vertices. Thus, we define the law of $L_0$ given $\x(O_0) = 0$ by
\begin{equation*}
\Prmp(\bfnull \in T_1 \,|\, \bfnull \in T_1 \cup \inner(T_0)) \Prmp(\ell_1 \in \cdot) + \Prmp(\bfnull \in \inner(T_0) \, |\,  \bfnull \in T_1 \cup \inner(T_0)) \Prmp(\ell_0 \in \cdot \,|\,  \bfnull \in \inner(T_0)).
\end{equation*}
In other words, we toss a coin with probability $\Prmp(\bfnull \in T_1 \,|\, \bfnull \in T_1 \cup \inner(T_0))$ for heads.
If the coin comes up heads, we sample the value of $L_0$ using an independent copy of $\ell_1$ (under $\Prmp$).
If the coin comes up tails, we sample the value of $L_0$ using an independent copy of $\ell_0$
(under $\Prmp$ given that $\bfnull \in \inner(T_0)$, this random variable satisfies the bound in Lemma \ref{Lem:trap length
distribution}(b)).
Additionally, if the coin comes up tails, we shift horizontally by a value $k\in\{1,\ldots,L_0\}$
according to the distribution under $\Prmp$ of the position of $\bfnull$ in $T_0$ given $\bfnull \in \inner(T_0)$.
This gives a new configuration $\tilde{\omega}$.
By construction, $\tilde{\omega} \disteq \omega$.

Slightly abusing notation, we write $\omegaprune$ for both $\omegaprune$
and the subset of $\tilde{\omega}$ corresponding to it.
We further write $V(\omegaprune)$ and $V(\tilde{\omega})$ for the
corresponding vertex sets.
Consequently, we write $u=v$ for vertices $u\in V(\omegaprune)$,
$v\in V(\tilde{\omega})$ if $v$ is the node in $\tilde{\omega}$ corresponding
to $u$ in $\omegaprune$.
Given $\omegaprune$ and $\tilde{\omega}$, we define a random walk
$(\mathcal{Y}_n)_{n\in\N_0}$ on $V(\omegaprune) \times V(\tilde{\omega})
\times \{-1,0,1\}$,
where the first and second component (up to random waiting times) behave like
$(Y_n^\mathrm{p})_{n\in\N_0}$ and $(Y_n)_{n\in\N_0}$, respectively,
and the third component exclusively acts as a memory of the directions
taken at certain nodes.
This is to ensure that $(\mathcal{Y}_n)_{n\in\N_0}$ is a Markov chain.

At each time $n\in\N_0$, first a candidate $\mathcal{Y}_{n+1}^\text{cand}
=
(\mathcal{Y}_{n+1,1}^\text{cand},\mathcal{Y}_{n+1,2}^\text{cand},\mathcal{Y}_{n+1,3}^\text{cand})$
for the next step is chosen and afterwards the chosen step is taken only
if the corresponding edges in $\omegaprune$ or $\tilde{\omega}$,
respectively, are open:
\begin{equation*}
\mathcal{Y}_{n+1,1} = 	\begin{cases}
					\mathcal{Y}^\text{cand}_{n+1,1}	&\text{if } \omegaprune(\<\mathcal{Y}_{n,1}, \mathcal{Y}^\text{cand}_{n+1,1}\>) =1,	\\
					\mathcal{Y}_{n,1}				&\text{otherwise},
					\end{cases} \quad
\mathcal{Y}_{n+1,2} = 	\begin{cases}
					\mathcal{Y}^\text{cand}_{n+1,2}	&\text{if } \tilde{\omega}(\<\mathcal{Y}_{n,1}\mathcal{Y}^\text{cand}_{n+1,1}\>) = 1, \\
					\mathcal{Y}_{n,2}				&\text{otherwise}
					\end{cases}
\end{equation*}
and $\mathcal{Y}_{n+1,3} = \mathcal{Y}^\text{cand}_{n+1,3}$.

We start at $\mathcal{Y}_0 = (\bfnull, \bfnull, 0)$ and give the
transition matrix of $(\mathcal{Y}_n)_{n\in\N_0}$ in a case-by-case
description depending on the position
$(u,v,w) \in V(\omegaprune) \times V(\tilde{\omega}) \times \{-1,0,1\}$ at
time $n$.

(1) If $u=v$ when regarding $\omegaprune$ as a subset of $\tilde{\omega}$,
and if $u \not= O_i$ for all $i\in\Z$, we let $(\mathcal{Y}_n)_{n\in\N_0}$
attempt to do exactly the same steps in its first two components.
In that case
\begin{equation*}
\mathcal{Y}^\text{cand}_{n+1} =	\begin{cases}
							(u+(1,0), v+(1,0), 0)	&	\text{with probability }
\frac{e^\lambda}{e^\lambda+1+e^{-\lambda}},		\\
							(u-(1,0), v-(1,0), 0)	&	\text{with probability }
\frac{e^{-\lambda}}{e^\lambda+1+e^{-\lambda}},	\\
							(u', v', 0)			&	\text{with probability }
\frac{1}{e^\lambda+1+e^{-\lambda}}.
							\end{cases}
\end{equation*}
Note that if $v$ is a trap entrance in $\tilde \omega$,
a step to the right by $(\mathcal{Y}_{n+1,1}^\mathrm{cand},
\mathcal{Y}_{n+1,2}^\mathrm{cand})$ induces a lazy step of
$(\mathcal{Y}_{k,1})_{k\in\N_0}$
whereas $(\mathcal{Y}_{k,2})_{k\in\N_0}$ moves into the trap.
In that case, as will be described in detail below,
$(\mathcal{Y}_{k,2} )_{k\in\N_0}$ will make an excursion into the trap
afterwards
whereas $(\mathcal{Y}_{k,1} )_{k\in\N_0}$ will stay put in $u$ until
$(\mathcal{Y}_{k,2} )_{k\in\N_0}$ returns to the trap entrance $v$.
Similarly, when a step of $(\mathcal{Y}_{k,1} )_{k\in\N_0}$ to the left
means moving to an obstacle,
$(\mathcal{Y}_{k,2})_{k\in\N_0}$ will then step onto a backbone node in
$\tilde{\omega} \setminus \omegaprune$.
In this case $(\mathcal{Y}_{k,1})_{k \in \N_0}$ will also stay put until
$(\mathcal{Y}_{k,2})_{k \in \N_0}$
reaches a node in $\tilde{\omega} \cap \omegaprune$.
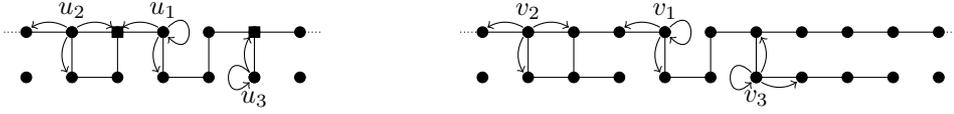
\begin{figure}[H]
\begin{center}
\begin{tikzpicture}[thin, scale=0.6,-,
                   shorten >=0pt+0.5*\pgflinewidth,
                   shorten <=0pt+0.5*\pgflinewidth,
                   every node/.style={circle, draw, fill = black, inner
sep = 0pt, minimum width =4 pt}
                   ]
\foreach \x in {-5,-4,-3,-2,-1,0,1}{
\foreach \y in {0,1}{
    \node at (\x,\y) {};
}}
\node[rectangle, draw, fill = black, minimum width = 4pt, minimum
height=4pt] at (-3,1) {};
\node[rectangle, draw, fill = black, minimum width = 4pt, minimum
height=4pt] at (0,1) {};
\foreach \x in {-5,-4,-3,-2,-1,0}{
\foreach \y in {1}{
    \ifdim \x pt < -2 pt \draw (\x,\y) -- (\x+1,\y);\fi
    \ifdim \x pt > -2 pt \draw (\x,\y) -- (\x+1,\y);\fi
  }}
\foreach \x in {-4,-2}{
	\draw (\x,0) -- (\x+1,0);
	}
\foreach \x in {-4,-3,-2,-1,0}{
	\draw (\x,0) -- (\x,1);
	}
\draw[densely dotted] (-5.5,1) -- (-5,1);
\draw[densely dotted] (1,1) -- (1.5,1);

\foreach \x in {5,6,7,8,9,10,11,12,13,14,15}
\foreach \y in {0,1}
    \node at (\x,\y) {};
\foreach \x in {5,6,7,8,10,11,12,13,14}{
\foreach \y in {1}{
    \ifdim \x pt < -1 pt \draw (\x,\y) -- (\x+1,\y);\fi
    \ifdim \x pt > -1 pt \draw (\x,\y) -- (\x+1,\y);\fi
  }}
\foreach \x in {6,7,9,11,12,13}{
\foreach \y in {0}{
	\draw (\x,\y) -- (\x+1,\y);
 }}
\foreach \x in {6,7,9,10,11}{
	\draw (\x,0) -- (\x,1);
	}
\draw[densely dotted] (4.5,1) -- (5,1);
\draw[densely dotted] (15,1) -- (15.5,1);

\node[label=above:$u_1$] at (-2,1) {};
\draw [->] (-2.1,1.1) to[bend right] (-2.9,1.1);
\draw [->] (-2.1,0.9) to[bend right] (-2.1,0.1);
\path (-1.9,1.1) edge [->, loop, distance=25,in=-45,out=45] (-1.9,0.9);
\node[label=above:$v_1$] at (9,1) {};
\draw [->] (8.9,1.1) to[bend right] (8.1,1.1);
\draw [->] (8.9,0.9) to[bend right] (8.9,0.1);
\path (9.1,1.1) edge [->, loop, distance=25,in=-45,out=45] (9.1,0.9);
\node[label=above:$u_2$] at (-4,1) {};
\draw [->] (-4.1,1.1) to[bend right] (-4.9,1.1);						
\draw [->] (-4.1,0.9) to[bend right] (-4.1,0.1);						
\draw [->] (-3.9,1.1) to [bend left] (-3.1,1.1); 					 	
\node[label=above:$v_2$] at (6,1) {};
\draw [->] (5.9,1.1) to[bend right] (5.1,1.1);
\draw [->] (5.9,0.9) to[bend right] (5.9,0.1);
\draw [->] (6.1,1.1) to[bend left] (6.9,1.1);
\node[label=below:$u_3$] at (0,0) {};
\draw [->] (-0.1,0.1) to[bend left] (-0.1,0.9);							
\path (-0.1,0.1) edge [->, loop, distance=25,in=-135,out=135] (-0.1,-0.1);
\node[label=below:$v_3$] at (11,0) {};
\draw [->] (11.1,-0.1) to[bend right] (11.9,-0.1);						
\draw [->] (11.1,0.1) to[bend right] (11.1,0.9);						
\path (10.9,0.1) edge [->, loop, distance=25,in=-135,out=135]
(10.9,-0.1);	

\end{tikzpicture}
\end{center}
\caption{The figure shows possible transitions on non-obstacle
backbone-nodes from $(u_1,v_1)$, $(u_2,v_2)$ and  $(u_3,v_3)$,
where $u_j$ in $\omegaprune$ `equals' $v_j$ in $\tilde \omega$.}
\end{figure}

(2) If $u=v$, but $u=O_i$ for some $i\in\N$, then the step in the first
component is taken according to the conductances $\cp$.
The second component mimics this, but with the additional option to move
right even if the first component does not.
This is to adjust the transition probabilities of the second component to
match those of $(Y_n)_{n \in \N_0}$.
If the first component moves right, we demand that the second component
leaves the coming trap piece at the right end,
which we encode in the third component.
Since we further want the walk in the second component to have the
same law as $(Y_n)_{n \in \N_0}$,
we have to make sure that in total, it leaves the trap piece at the right
resp.\ left end with the correct probability.
These restrictions lead to a system of linear equations for the transition
probabilities whose solution is given as follows.
\begin{equation*}
\mathcal{Y}^\text{cand}_{n+1}	
= \begin{cases} (u+(1,0), v+(1,0), 1) &\text{with probability }
\frac{e^\lambda(1-e^{-2\lambda})}{e^\lambda(1-e^{-2\lambda})+1+e^{-\lambda}}
=\frac{e^\lambda-e^{-\lambda}}{e^\lambda+1},	\\
(u-(1,0), v-(1,0), 0) &\text{with probability }
\frac{e^{-\lambda}}{e^\lambda+1+e^{-\lambda}},	\\
(u', v', 0) &\text{with probability } 
\frac{1}{e^\lambda+1+e^{-\lambda}}, \\
(u-(1,0), v+(1,0), 1) &\text{with probability }	
\frac{e^{-\lambda}}{1+e^{-\lambda}}\big( \mathsf{e}_{L_i+1}' -
\frac{e^\lambda-e^{-\lambda}}{e^\lambda+1} \big),	\\
(u-(1,0), v+(1,0), -1) &\text{with probability }
e^{-\lambda}\big( \frac{1}{1+e^{-\lambda}} - \frac{1}{e^\lambda+1+e^{-\lambda}}
- \frac{1}{1+e^{-\lambda}}\mathsf{e}_{L_i+1}' \big), \\
(u', v+(1,0), 1) &\text{with probability }
\frac{1}{1+e^{-\lambda}}\big( \mathsf{e}_{L_i+1}' - \frac{e^\lambda-e^{-\lambda}}{e^\lambda+1} \big), \\
(u', v+(1,0), -1) &\text{with probability }
\frac{1}{1+e^{-\lambda}} - \frac{1}{e^\lambda+1+e^{-\lambda}}
- \frac{1}{1+e^{-\lambda}}\mathsf{e}_{L_i+1}',
\end{cases}
\end{equation*}
where $L_i$ is the length of the trap right of $v$
and
\begin{equation*}
\mathsf{e}_m' \defeq \frac{e^\lambda}{e^\lambda+1+e^{-\lambda}}
P_{m,\lambda}^1(\sigma_m<\sigma_0) =
\frac{e^\lambda}{e^\lambda+1+e^{-\lambda}} \frac{1 -
e^{-2\lambda}}{1-e^{-2\lambda m}}
\end{equation*}
is the probability that the biased random walk $(S_n')_{n\in\N_0}$ on $\Z$ starting from $0$ first
makes a step to the right and then hits $m$ before $0$.
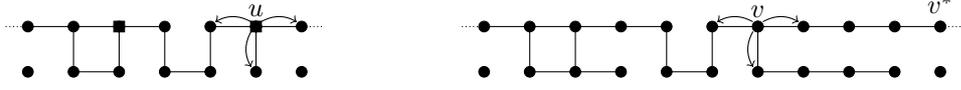
\begin{figure}[H]
\begin{center}
\begin{tikzpicture}[thin, scale=0.6,-,
                   shorten >=0pt+0.5*\pgflinewidth,
                   shorten <=0pt+0.5*\pgflinewidth,
                   every node/.style={circle, draw, fill = black, inner
sep = 0pt, minimum width =4 pt}
                   ]
\foreach \x in {-5,-4,-3,-2,-1,0,1}{
\foreach \y in {0,1}{
    \node at (\x,\y) {};
}}
\node[rectangle, draw, fill = black, minimum width = 4pt, minimum
height=4pt] at (-3,1) {};
\node[rectangle, draw, fill = black, minimum width = 4pt, minimum
height=4pt] at (0,1) {};
\foreach \x in {-5,-4,-3,-2,-1,0}{
\foreach \y in {1}{
    \ifdim \x pt < -2 pt \draw (\x,\y) -- (\x+1,\y);\fi
    \ifdim \x pt > -2 pt \draw (\x,\y) -- (\x+1,\y);\fi
  }}
\foreach \x in {-4,-2}{
	\draw (\x,0) -- (\x+1,0);
	}
\foreach \x in {-4,-3,-2,-1,0}{
	\draw (\x,0) -- (\x,1);
	}
\draw[densely dotted] (-5.5,1) -- (-5,1);
\draw[densely dotted] (1,1) -- (1.5,1);

\foreach \x in {5,6,7,8,9,10,11,12,13,14,15}
\foreach \y in {0,1}
    \node at (\x,\y) {};
\foreach \x in {5,6,7,8,10,11,12,13,14}{
\foreach \y in {1}{
    \ifdim \x pt < -1 pt \draw (\x,\y) -- (\x+1,\y);\fi
    \ifdim \x pt > -1 pt \draw (\x,\y) -- (\x+1,\y);\fi
  }}
\foreach \x in {6,7,9,11,12,13}{
\foreach \y in {0}{
	\draw (\x,\y) -- (\x+1,\y);
 }}
\foreach \x in {6,7,9,10,11}{
	\draw (\x,0) -- (\x,1);
	}
\draw[densely dotted] (4.5,1) -- (5,1);
\draw[densely dotted] (15,1) -- (15.5,1);

\node[label=above:$u$] at (0,1) {};
\draw [->] (0.1,1.1) to[bend left] (0.9,1.1);								
\draw [->] (-0.1,1.1) to[bend right] (-0.9,1.1);							
\draw [->] (-0.1,0.9) to[bend right] (-0.1,0.1);							

\node[label=above:$v$] at (11,1) {};
\node[label=above:$v^*$] at (15,1) {};
\draw [->] (11.1,1.1) to[bend left] (11.9,1.1);							
\draw [->] (10.9,1.1) to[bend right] (10.1,1.1);							
\draw [->] (10.9,0.9) to[bend right] (10.9,0.1);							
\end{tikzpicture}
\end{center}
\caption{Transitions from obstacles. Depending on the value of $\mathcal{Y}_{n+1,3}$,
after a step to the right it is already determined whether the random walk
on $\tilde{\omega}$ hits the boundary of the trap piece
at $v$ or $v^*$.}
\end{figure}

(3) If $v$ is in the interior of the backbone part of a trap piece in $\tilde{\omega}$ (and thus not in $\omegaprune$),
then we write $L_v$ for the length of the corresponding trap.
In this case,
the first component of $(\mathcal{Y}_n)_{n\in\N_0}$ stays put
while the second component moves in the trap piece
with transition probabilities according to the biased random walk
$(Y_n)_{n\in\N_0}$, possibly conditioned on the event that the boundary of the trap piece
is first hit at the left- or rightmost end, respectively.
Let $p_{k,0}$, $p_{k,-1}$, $p_{k,1}$ be the transition matrices of the
lazy biased random walk $(S_n)_{n\in\N_0}$ on $\{0,...,k\}$
(which steps to the right, steps to the left or stays put with probability proportional to $e^\lambda$, $e^{-\lambda}$ and $1$, respectively)
and the lazy biased random walk on $\{0,...,k\}$ conditioned on
$\{\sigma_0<\sigma_k\}$ resp.\ $\{\sigma_0>\sigma_k\}$,
where $\sigma_j \defeq \inf\{n \in \N_0: S_n=j\}$.
Then we set
\begin{equation*}
\mathcal{Y}^\text{cand}_{n+1}	=	\begin{cases}
							(u,v+(1,0),w)				&\text{with probability }
p_{L_v+1,w}(\x_v,\x_v+1),	\\
							(u,v-(1,0),w)				&\text{with probability }
p_{L_v+1,w}(\x_v,\x_v-1),	\\
							(u,v',w)					&\text{with probability } p_{L_v+1,w}(\x_v,\x_v),
							\end{cases}
\end{equation*}
where $\x_v \in \{1,...,L_v\}$ is the relative horizontal position of $v$ in the trap piece.

\begin{figure}[H]
\begin{center}
\begin{tikzpicture}[thin, scale=0.6,-,
                   shorten >=0pt+0.5*\pgflinewidth,
                   shorten <=0pt+0.5*\pgflinewidth,
                   every node/.style={circle, draw, fill = black, inner
sep = 0pt, minimum width =4 pt}
                   ]
\foreach \x in {-5,-4,-3,-2,-1,0,1}{
\foreach \y in {0,1}{
    \node at (\x,\y) {};
}}
\node[rectangle, draw, fill = black, minimum width = 4pt, minimum
height=4pt] at (-3,1) {};
\node[rectangle, draw, fill = black, minimum width = 4pt, minimum
height=4pt] at (0,1) {};
\foreach \x in {-5,-4,-3,-2,-1,0}{
\foreach \y in {1}{
    \ifdim \x pt < -2 pt \draw (\x,\y) -- (\x+1,\y);\fi
    \ifdim \x pt > -2 pt \draw (\x,\y) -- (\x+1,\y);\fi
  }}
\foreach \x in {-4,-2}{
	\draw (\x,0) -- (\x+1,0);
	}
\foreach \x in {-4,-3,-2,-1,0}{
	\draw (\x,0) -- (\x,1);
	}
\draw[densely dotted] (-5.5,1) -- (-5,1);
\draw[densely dotted] (1,1) -- (1.5,1);

\foreach \x in {5,6,7,8,9,10,11,12,13,14,15}
\foreach \y in {0,1}
    \node at (\x,\y) {};
\foreach \x in {5,6,7,8,10,11,12,13,14}{
\foreach \y in {1}{
    \ifdim \x pt < -1 pt \draw (\x,\y) -- (\x+1,\y);\fi
    \ifdim \x pt > -1 pt \draw (\x,\y) -- (\x+1,\y);\fi
  }}
\foreach \x in {6,7,9,11,12,13}{
\foreach \y in {0}{
	\draw (\x,\y) -- (\x+1,\y);
 }}
\foreach \x in {6,7,9,10,11}{
	\draw (\x,0) -- (\x,1);
	}
\draw[densely dotted] (4.5,1) -- (5,1);
\draw[densely dotted] (15,1) -- (15.5,1);

\path (-0.1,1.1) edge [->, loop, distance=25,in=45,out=135] (0.1,1.1);  
\path (-1.1,1.1) edge [->, loop, distance=25,in=45,out=135] (-0.9,1.1);  
\path (0.9,1.1) edge [->, loop, distance=25,in=45,out=135] (1.1,1.1);  
\path (-0.1,-0.1) edge [->, loop, distance=25,in=135,out=225] (-0.1,0.1); 
\draw [->] (13.1,1.1) to[bend left] (13.9,1.1);						
\draw [->] (12.9,1.1) to[bend right] (12.1,1.1);					
\path (12.9,0.9) edge [->, loop, distance=25,in=-45,out=-135] (13.1,0.9); 
\node[label=above:$v_1$] at (11,1) {};
\node[label=above:$v_2$] at (15,1) {};

\end{tikzpicture}
\end{center}
\caption{Transitions in the backbone part of trap pieces.
If $\mathcal{Y}_{n,3} \in \{-1,1\}$, then it is predetermined
that the walk hits the boundary of the trap piece at $v_1$ or $v_2$, respectively.}
\end{figure}
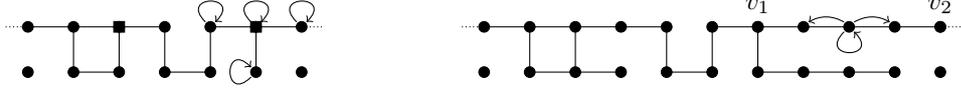

(4) If $v$ is a trap node in $\tilde{\omega}$, the first component of
$(\mathcal{Y}_n)_{n\in\N_0}$ stays put
while the second component moves inside the trap with transition probabilities according to the biased
random walk $(Y_n)_{n\in\N_0}$.
That is,
\begin{equation*} 
\mathcal{Y}^\text{cand}_{n+1}
=\begin{cases}
(u,v+(1,0),0) &\text{with probability } \frac{e^\lambda}{e^\lambda+1+e^{-\lambda}} \\
(u,v-(1,0),0) &\text{with probability } \frac{e^{-\lambda}}{e^\lambda+1+e^{-\lambda}} \\
(u,v',0) &\text{with probability } \frac{1}{e^\lambda+1+e^{-\lambda}}
\end{cases}.
\end{equation*}

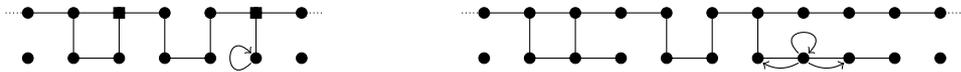
\begin{figure}[H]
\begin{center}
\begin{tikzpicture}[thin, scale=0.6,-,
                   shorten >=0pt+0.5*\pgflinewidth,
                   shorten <=0pt+0.5*\pgflinewidth,
                   every node/.style={circle, draw, fill = black, inner
sep = 0pt, minimum width =4 pt}
                   ]
\foreach \x in {-5,-4,-3,-2,-1,0,1}{
\foreach \y in {0,1}{
    \node at (\x,\y) {};
}}
\node[rectangle, draw, fill = black, minimum width = 4pt, minimum
height=4pt] at (-3,1) {};
\node[rectangle, draw, fill = black, minimum width = 4pt, minimum
height=4pt] at (0,1) {};
\foreach \x in {-5,-4,-3,-2,-1,0}{
\foreach \y in {1}{
    \ifdim \x pt < -2 pt \draw (\x,\y) -- (\x+1,\y);\fi
    \ifdim \x pt > -2 pt \draw (\x,\y) -- (\x+1,\y);\fi
  }}
\foreach \x in {-4,-2}{
	\draw (\x,0) -- (\x+1,0);
	}
\foreach \x in {-4,-3,-2,-1,0}{
	\draw (\x,0) -- (\x,1);
	}
\draw[densely dotted] (-5.5,1) -- (-5,1);
\draw[densely dotted] (1,1) -- (1.5,1);

\foreach \x in {5,6,7,8,9,10,11,12,13,14,15}
\foreach \y in {0,1}
    \node at (\x,\y) {};
\foreach \x in {5,6,7,8,10,11,12,13,14}{
\foreach \y in {1}{
    \ifdim \x pt < -1 pt \draw (\x,\y) -- (\x+1,\y);\fi
    \ifdim \x pt > -1 pt \draw (\x,\y) -- (\x+1,\y);\fi
  }}
\foreach \x in {6,7,9,11,12,13}{
\foreach \y in {0}{
	\draw (\x,\y) -- (\x+1,\y);
 }}
\foreach \x in {6,7,9,10,11}{
	\draw (\x,0) -- (\x,1);
	}
\draw[densely dotted] (4.5,1) -- (5,1);
\draw[densely dotted] (15,1) -- (15.5,1);

\path (-0.1,-0.1) edge [->, loop, distance=25,in=135,out=225] (-0.1,0.1); 
\draw [->] (12.1,-0.1) to[bend right] (12.9,-0.1);							
\draw [->] (11.9,-0.1) to[bend left] (11.1,-0.1);							
\path (11.9,0.1) edge [->, loop, distance=25,in=45,out=135] (12.1,0.1);  

\end{tikzpicture}
\end{center}
\caption{Transitions in the dead end part of trap pieces}
\end{figure}

(5) Finally, when $v\in \tilde{\omega} \cap \omegaprune$, but the positions of
the two components of $(\mathcal{Y}_n)_{n\in\N_0}$ do not correspond, the
second component stays put,
while the first component moves with transition probabilities given by the
conductances $\cp$:
\begin{equation*}
\mathcal{Y}^\text{cand}_{n+1}	
= \begin{cases}
(u+(1,0),v,0) &\text{with probability proportional to } \cp(\<u,u+(1,0)\>), \\
(u-(1,0),v,0) &\text{with probability proportional to } \cp(\<u,u-(1,0)\>), \\
(u',v,0) &\text{with probability proportional to } \cp(\<u,u'\>). \\
\end{cases}
\end{equation*}

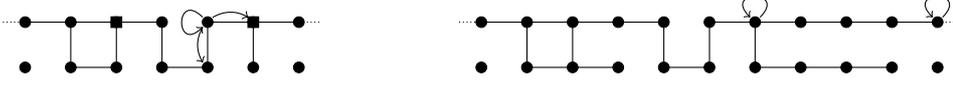
\begin{figure}[H]
\begin{center}
\begin{tikzpicture}[thin, scale=0.6,-,
                   shorten >=0pt+0.5*\pgflinewidth,
                   shorten <=0pt+0.5*\pgflinewidth,
                   every node/.style={circle, draw, fill = black, inner
sep = 0pt, minimum width =4 pt}
                   ]
\foreach \x in {-5,-4,-3,-2,-1,0,1}{
\foreach \y in {0,1}{
    \node at (\x,\y) {};
}}
\node[rectangle, draw, fill = black, minimum width = 4pt, minimum
height=4pt] at (-3,1) {};
\node[rectangle, draw, fill = black, minimum width = 4pt, minimum
height=4pt] at (0,1) {};
\foreach \x in {-5,-4,-3,-2,-1,0}{
\foreach \y in {1}{
    \ifdim \x pt < -2 pt \draw (\x,\y) -- (\x+1,\y);\fi
    \ifdim \x pt > -2 pt \draw (\x,\y) -- (\x+1,\y);\fi
  }}
\foreach \x in {-4,-2}{
	\draw (\x,0) -- (\x+1,0);
	}
\foreach \x in {-4,-3,-2,-1,0}{
	\draw (\x,0) -- (\x,1);
	}
\draw[densely dotted] (-5.5,1) -- (-5,1);
\draw[densely dotted] (1,1) -- (1.5,1);

\foreach \x in {5,6,7,8,9,10,11,12,13,14,15}
\foreach \y in {0,1}
    \node at (\x,\y) {};
\foreach \x in {5,6,7,8,10,11,12,13,14}{
\foreach \y in {1}{
    \ifdim \x pt < -1 pt \draw (\x,\y) -- (\x+1,\y);\fi
    \ifdim \x pt > -1 pt \draw (\x,\y) -- (\x+1,\y);\fi
  }}
\foreach \x in {6,7,9,11,12,13}{
\foreach \y in {0}{
	\draw (\x,\y) -- (\x+1,\y);
 }}
\foreach \x in {6,7,9,10,11}{
	\draw (\x,0) -- (\x,1);
	}
\draw[densely dotted] (4.5,1) -- (5,1);
\draw[densely dotted] (15,1) -- (15.5,1);

\draw [->] (-0.9,1.1) to[bend left] (-0.1,1.1);								
\draw [->] (-1.1,0.9) to[bend right] (-1.1,0.1);							
\path (-1.1,1.1) edge [->, loop, distance=25,in=225,out=135] (-1.1,0.9);  

\path (11.1,1.1) edge [->, loop, distance=25,in=135,out=45] (10.9,1.1);  
\path (15.1,1.1) edge [->, loop, distance=25,in=135,out=45] (14.9,1.1);  

\end{tikzpicture}
\end{center}
\caption{Transitions on the backbone when coordinates do not coincide.
In this case, the walk on $\tilde{\omega}$ waits at a trap end
or a vertex opposite a trap entrance. This vertex must be passed by the walk on $\omegaprune$
provided that this walk is transient to the right.
The walk on $\omegaprune$ pauses until the walk on $\tilde{\omega}$ hits its position.}
\end{figure}

We write $\Prmp'$ for the distribution of the environment
$(\omegaprune,\tilde{\omega})$ and
$P_{\omegaprune,\tilde{\omega},\lambda}'$ for the quenched law of
$(\mathcal{Y}_n)_{n\in\N_0}$ as described above.
With these, we define a measure $\Prob'$ on $(\{0,1\}^E \times \{0,1\}^E)
\times (V^2\times \{-1,0,1\})^{\N_0}$,
endowed with the product $\sigma$-Algebra,
by
\begin{equation*}
\Prob'(A\times B) \defeq \int_A P_{\omegaprune,\tilde{\omega},\lambda}'(B) \,\Prmp'(\mathrm{d}(\omegaprune,\tilde{\omega})).
\end{equation*}
Sometimes, the walks on $\omegaprune$ and $\tilde{\omega}$ are at different positions (when $\omegaprune$ is embedded in $\tilde{\omega}$).
Then, depending on the particular situation, one of the walks waits while the other moves until they meet again.
The times at which each of the walks moves without being forced to hold as described above are collected in the following sets:
\begin{align*}
	N_1	&\defeq	\{ n\in\N_0: \mathcal{Y}_{n,2} \text{ is at a vertex in } \tilde{\omega} \text{ corresponding to a vertex in } \omegaprune \},		\\
	N_2	&\defeq \{ n\in\N_0: \mathcal{Y}_{n,1} = \mathcal{Y}_{n,2} \} \cup \{n\in\N_0: \mathcal{Y}_{n,2} \text{ is in the interior of a trap piece}\}.
\end{align*}
Let $(s_{1,k})_{k\in\N}$ resp.\ $(s_{2,k})_{k\in\N}$ be enumerations of
$N_1$ resp.\ $N_2$ in ascending order.
Then the following processes coincide in law with $(Y_n^\mathrm{p})_{n\in\N_0}$ and $(Y_n)_{n\in\N_0}$, respectively.
More precisely, with
\begin{equation*}
(\mathcal{Y}_n^\mathrm{p})_{n\in\N_0}	\defeq	(\mathcal{Y}_{s_{1,n},1})_{n\in\N_0},		\quad
	(\tilde Y_n)_{n\in\N_0}		\defeq	(\mathcal{Y}_{s_{2,n},2})_{n\in\N_0}
\end{equation*}
the following lemma holds.

\begin{lemma} \label{Lem:law of pruned and extended pruned walk}
We have
\begin{equation*}
( \mathcal{Y}_n^\mathrm{p} )_{n\in\N_0} 	\disteq 	(Y_n^\mathrm{p})_{n\in\N_0},
\quad	 ( \tilde Y_n )_{n\in\N_0}		\disteq	(Y_n)_{n\in\N_0}.
\end{equation*}
\end{lemma}

\begin{proof}
Since $(\mathcal{Y}_n^\mathrm{p} )_{n\in\N_0}$ and $(Y_n^\mathrm{p})_{n\in\N_0}$ are defined on the same environment,
and the environments of $( \tilde Y_n )_{n\in\N_0}$ and $(Y_n)_{n\in\N_0}$ are
identically distributed by construction,
it suffices to check the quenched transition probabilities of $( \tilde Y_n )_{n\in\N_0}$ and $( \mathcal{Y}_n^\mathrm{p} )_{n\in\N_0}$, respectively.
One can check that the transition probabilities of $(\mathcal{Y}_n^\mathrm{p})_{n\in\N_0}$ coincide with those of
$(Y_n^\mathrm{p})_{n\in\N_0}$,
thus the equality in law of $(Y_n^\mathrm{p})_{n\in\N_0}$ and
$(\mathcal{Y}_n^\mathrm{p})_{n\in\N_0}$ follows from the Markov property
of $(\mathcal{Y}_n)_{n\in\N_0}$.
For $(\tilde Y_n)_{n\in\N_0}$, at most nodes this is also obvious except for
transitions at obstacles and inside trap pieces.
However, it suffices to show that on obstacles, steps into the different
directions are taken with the correct probability
and that excursions on the following trap pieces end on the left resp.\
right end with the correct probability,
i.e., that $(\mathcal{Y}_{n,3})_{n\in\N_0}$ takes value $-1$ or $1$ with
the correct probability.
This amounts to a system of linear equations which is solved by the
transition probabilities defined under (2).
The result now also follows from the Markov property of $(\mathcal{Y}_n)_{n\in\N_0}$.
\end{proof}

From now on, all results concerning $(Y_n)_{n\in\N_0}$ will be discussed
in terms of the process $(\tilde Y_n)_{n\in\N_0}$ under $\Prob'$.
To ease notation, we shall write $(Y_n)_{n\in\N_0}$ and $\Prob$ for
$(\tilde Y_n)_{n\in\N_0}$ and $\Prob'$, respectively.
We shall also write $\ell_i$ though technically referring to $L_i$.
Consequently, we shall not distinguish between
$(Y_n^\mathrm{p})_{n\in\N_0}$ and $(\mathcal{Y}_n^\mathrm{p})_{n\in\N_0}$
nor between $\omega$ and $\tilde{\omega}$.

\begin{lemma}\label{Lem:transience pruned walk}
For $\lambda > \lambda^*\defeq \frac{ \log(2) }{2}$, especially for
$\lambda \geq \frac{\lambdacrit}{2}$, it holds that $\lim_{n \to \infty}
\x(Y_n^\mathrm{p}) = \infty$ a.\,s.
\end{lemma}

The proof of the lemma is very similar to that of Proposition 3.1 in
\cite{Axelson-Fisk+H"aggstr"om:2009b}.
We include it for completeness.

\begin{proof}
It is sufficient to show that $\bfnull$ is a transient state for the
biased random walk on $V(\omegaprune)$.
We use electrical network theory.
Write $\mathcal{R}^{\mathrm{p}}(\bfnull \leftrightarrow \infty)$
for the effective resistance between $\bfnull$ and $+\infty$ in the random
conductance model on $\omegaprune$ with conductances $\cp(e)$ for $e \in
E$ with $\omegaprune(e)=1$.
Using Thomson's Principle \cite[Theorem 9.10]{Levin+Peres+Wilmer:2009}, we
infer
\begin{equation*}
\mathcal{R}^{\mathrm{p}}(\bfnull \leftrightarrow \infty) \leq
\mathcal{E}^{\mathrm{p}}(\theta)
\end{equation*}
for all unit flows $\theta$ from $\bfnull$ to $\infty$ where
$\mathcal{E}^{\mathrm{p}}(\theta)$ is the energy of the flow $\theta$.
Here a flow $\theta$ from $u$ to $\infty$ is a mapping $\theta:
V(\omegaprune) \times V(\omegaprune) \to \R$ satisfying the properties
\begin{enumerate}[(i)]
	\item	$\theta(v,w) = 0$ unless there is an open edge connecting $v$ and
$w$ in $\omegaprune$;
	\item	$\theta(v,w) = -\theta(w,v)$ for all $v,w \in V(\omegaprune)$;
	\item	$\sum_{w \in V(\omegaprune)} \theta(v,w) = \1_{\{u\}}(v)$ for all
$v \in V(\omegaprune)$.
\end{enumerate}
The energy of the flow $\theta$ is $\mathcal{E}^{\mathrm{p}}(\theta) =
\sum_{e: \omegaprune(e)=1} \theta(e)^2/\cp(e)$
where $\theta(e)^2 = \theta(v,w)^2$ if $e=\<v,w\>$.
Since there are no traps in $\omegaprune$, there exists an infinite open
self-avoiding path $P = (e_1,e_2,\ldots)$ connecting $\bfnull$ with
$\infty$.
This path never backtracks in the sense that the sequence of
$\x$-coordinates of the vertices on this path is nondecreasing.
Now define a flow $\theta$ from $\bfnull$ to $\infty$ by pushing a unit
current through $P$.
More precisely, if $e_n = \langle u_{n-1},u_n\rangle$ with $u_0 \defeq
\bfnull$, then let $\theta(u_{n-1},u_n) = 1 = -\theta(u_n,u_{n-1})$ for
all $n \in \N$
and $\theta(v,w)=0$ whenever $\<v,w\>$ is not on the path $P$.
For every $\x$-level $n \in \N_0$ there is at most one edge $e$ in $P$
connecting the two vertices with $\x$-value $n$.
The resistance of this edge is bounded by $r^{\mathrm{p}}(e) \leq
e^{-2\lambda n} (1-e^{-2\lambda})^{-p(n)}$
where $p(n)$ is the number of obstacles with $\x$-value $<n$. There are at
most $n$ such obstacles. Therefore,
$r^{\mathrm{p}}(e) \leq e^{-2\lambda n} (1-e^{-2\lambda})^{-n}$.
Further, for every $n \in \N$, there is exactly one edge on $P$ leading
from a vertex with $\x$-value $n-1$ to $\x$-value $n$.
The resistance of this edge is bounded by $r^{\mathrm{p}}(e) \leq
e^{-\lambda (2n-1)} (1-e^{-2\lambda})^{-p(n)} \leq e^{-\lambda (2n-1)}
(1-e^{-2\lambda})^{-n}$.
Consequently, the energy $\mathcal{E}^{\mathrm{p}}(\theta)$ is bounded by
\begin{align*}
\mathcal{E}^{\mathrm{p}}(\theta)
&=		\sum_{e \in P} \theta(e)^2 r^{\mathrm{p}}(e)
\leq		1+ \sum_{n=1}^\infty (e^{-\lambda (2n-1)} + e^{-2\lambda n})
(1-e^{-2\lambda})^{-n}
\leq 	1+ 2 e^\lambda \sum_{n=1}^\infty
\Big(\frac{e^{-2\lambda}}{1-e^{-2\lambda}}\Big)^n.
\end{align*}
The latter series is finite iff
$\frac{e^{-2\lambda}}{1-e^{-2\lambda}} < 1$ or, equivalently,
$\lambda > \frac{\log(2)}{2} \eqdef \lambda^*$.
Comparing this with $\lambdacrit/2$, for which we have an explicit formula
in terms of $p$ given in Proposition \ref{Prop:transience of Y_n} with
unique minimizer $p=1/2$,
we have
\begin{align*}
\tfrac{\lambdacrit}{2}
&\geq	\tfrac{ \lambdacrit(1/2)}{2}
=	\tfrac{1}{4} \log\Big(\tfrac{4}{ 3-\sqrt{5}}\Big)
=	\tfrac{1}{2} \log\Big(\tfrac{2}{\sqrt{ 3 - \sqrt{5}}}\Big)
> 	\tfrac{\log(2)}{2}
=	\lambda^*.
\end{align*}
\end{proof}

It also follows from the proof of Lemma \ref{Lem:transience pruned walk}
that for $u \in \omegaprune$ and $\lambda \geq \lambdacrit/2$,
the escape probability at $u$, i.e., the probabilty to leave $u$ and never return, is uniformly bounded from below.
For $u \in \omegaprune$, let $\sigma_u^\mathrm{p} \defeq \inf\{n>0: Y_n^\mathrm{p} = u \}$. Also let
$\mathcal{R}^{\mathrm{p}}(u \leftrightarrow \infty)$ and $\cp(u)$ be the
effective resistance between $u$ and $+\infty$ and the sum of conductances
of all incident edges at $u$, respectively, in the random conductance
model on $\omegaprune$ with conductances $\cp(e)$ for $e \in E$ with
$\omegaprune(e)=1$. Then pushing a unit current from $u$ to $+\infty$ as
in the proof of Lemma \ref{Lem:transience pruned walk}, we get
\begin{align}\label{eqn:p_esc^p}
P_{\omega,\lambda}^u\big( \sigma_u^\mathrm{p} = \infty \big)
&=		\frac{1}{c^\mathrm{p}(u) \mathcal{R}^{\mathrm{p}}(u\leftrightarrow \infty)}		\notag\\
&\geq	\frac{1}{3e^{(2\x(u)+1)\lambda}(1-e^{-2\lambda})^{p(u)} e^{-2\lambda\x(u)}(1-e^{-2\lambda})^{-p(u)}
		\big(1+2e^{2\lambda}\sum_{n=1}^\infty \big(\frac{e^{-2\lambda}}{1-e^{-2\lambda}}\big)^{\!n} \big) }	\notag \\
&= \frac{1}{3 e^\lambda \big(1+2e^{2\lambda}\sum_{n=1}^\infty \big(\frac{e^{-2\lambda}}{1-e^{-2\lambda}}\big)^{\!n}\big) }
> 0.
\end{align}

Let $R_1^\mathrm{p}, R_2^\mathrm{p}, \ldots$ be an enumeration
from left to right of the pre-regeneration points in $\omegaprune$
which are visited exactly once by $(Y_n^\mathrm{p} )_{n\in\N_0}$.
Further, let $\rho_0^\mathrm{p}=0$ and $\rho_n^\mathrm{p} \defeq 
\x(R_n^\mathrm{p})$ for $n \in \N$.
Finally, for $n \in \N$, let $\tau_n^\mathrm{p}$ be the unique time $k$ 
with $X_{k}^\mathrm{p} = \rho_n^\mathrm{p}$.
We refer to the $R_n^\mathrm{p}$'s and $\tau_n^\mathrm{p}$'s as 
regeneration points and times, respectively,
of the pruned walk.

\begin{lemma}\label{Lem:pruned regeneration points}
With $\Prob$-probability $1$, there exist infinitely many regeneration 
points of
$( Y_n^\mathrm{p} )_{n\in\N_0}$.
\end{lemma}
\begin{proof}
This can be proven along exactly the same lines as for $(Y_n)_{n\in\N_0}$
in \cite[Lemma 5.1]{Axelson-Fisk+H"aggstr"om:2009b}, as the argument there
only relies on a uniform lower bound on the escape probability at any
pre-regeneration point $u$. Here, \eqref{eqn:p_esc^p} gives this estimate.
\end{proof}

\begin{lemma} \label{Lem:moment estimate for tau_1^p}
Let $\lambda > \lambda^*$. Then there exists $\delta >0$ such that
\begin{equation*}
\E^\circ \big[ e^{\delta (\rho_1^\mathrm{p} - \min_{j \in \N}\, 
	\x(Y_j^\mathrm{p}))} \big] < \infty.
\end{equation*}
Furthermore, $\E^\circ [(\tau_1^\mathrm{p})^\kappa ] < \infty$ for any 
$\kappa>0$.
\end{lemma}

Both statements still hold true when $\E^\circ$ is replaced by $\E$.

\begin{proof}
We shall only give an informal description of the proof as the details of 
it can be adapted from the proofs of Lemmas 6.3 through 6.5 in \cite{Gantert+al:2018}.
The basic idea is to consider the walk $(Y_n^\mathrm{p})_{n\in\N_0}$ at 
fresh points.
The first fresh point $F_1^\mathrm{p}$ is the first pre-regeneration 
point to the right of the origin visited by the walk 
$(Y_n^\mathrm{p})_{n\in\N_0}$.
If after the first visit to this fresh point the random walk never 
returns to it, then $F_1^\mathrm{p} = R_1^\mathrm{p}$.
Otherwise, the random walk will return to $F_1^\mathrm{p}$. In this 
case, the second fresh point $F_2^\mathrm{p}$ is the first 
pre-regeneration point to the right of $F_1^\mathrm{p}$
that has not been visited by the random walk before hitting 
$F_1^\mathrm{p}$ for the second time, and so on (see Lemma 6.4 in 
\cite{Gantert+al:2018}
for the construction for the original walk).
By the strong Markov property (for the walk and the cluster where
a cycle to the right of the origin in the pruned cluster is revealed upon the 
first visit of the walk to this cycle), the distances between two fresh 
points are i.i.d.\ given they are finite.
Using the uniform bound on the resistance to $+\infty$ given in the 
proof of Lemma \ref{Lem:transience pruned walk},
valid for $\lambda > \lambda^*$, the walk will visit at most a geometric 
number of fresh points before hitting a fresh point
from which it escapes to $+\infty$ without ever returning to that point.
If, on the other hand, the distance between two consecutive fresh points, a left and a right one,
is large, say $\geq 2m$, then there are two options.
Either the walk made an excursion of length at least $m$ to the right between the first two visits
of the walk to the left fresh point, or there is no pre-regeneration point
on the percolation cluster from distance $m$ to distance $2m$ to the 
right of the left fresh point.
Both possibilities are exponentially unlikely in $m$.
The first one because it requires the walk to backtrack at least $m$ 
steps to the left, which has probabilty bounded by a constant times
$(e^{-2\lambda}/(1-e^{-2\lambda}))^{-m}$ (adapt the proof of Lemma 6.3 
in \cite{Gantert+al:2018} with the new conductances to see this).
The second one because of the Markov property of the original 
percolation cluster $\omega$,
which implies that when exploring the cluster from the left to the 
right, at any point, the next pre-regeneration point to the right is 
only a geometric distance away.
Consequently, $\rho_1^\mathrm{p}$ can be bounded from above by a 
geometric number of independent random variables
all stochastically bounded by a nonnegative integer-valued random 
variable with some finite exponential moment.
From this, standard large deviation estimates imply that 
$\rho_1^\mathrm{p}$ has exponentially decaying tails.
The proof of $\E^\circ [(\tau_1^\mathrm{p})^\kappa ] < \infty$ for 
arbitrary $\kappa>0$ can be adapted from the proof of Lemma 6.5 in 
\cite{Gantert+al:2018},
a brute-force estimate which carries over immediately.
\end{proof}

\subsection{Proof of Proposition \ref{Prop:tail estimate}}
We are now ready to give the proof of the tail result for the regeneration times.

\begin{proof}[Proof of Proposition \ref{Prop:tail estimate}]
For each $n \in \N$, we have
\begin{align*}
\Prob(\tau_2 - \tau_1 \geq n)
&\leq \Prob\big((\tau_2-\tau_1)^\B \geq n/2\big)
+ \Prob\big((\tau_2-\tau_1)^\mathrm{traps} \geq n/2\big).
\end{align*}
The time spent on the backbone can be neglected due to
Lemma \ref{Lem:backbone time}. We now estimate the time spent in traps.
From Lemma 4.1 in \cite{Gantert+al:2018}, we infer
\begin{equation*}
\Prob\big((\tau_2-\tau_1)^\mathrm{traps} \geq n\big)
= \Prob^\circ\big(\tau_1^\mathrm{traps} \geq n | X_k \geq 1 \text{ for all } k \in \N\big).
\end{equation*}
If $\bfnull$ is a pre-regeneration point (or just connected to $+\infty$
via a path that does not visit vertices with $\x$-coordinate strictly
smaller than $0$),
the argument that leads to (24) in \cite{Axelson-Fisk+H"aggstr"om:2009b}
gives
\begin{equation*}
P_{\omega,\lambda}(Y_n \not = \bfnull \text{ for all } n \in \N)
\geq \frac{(\sum_{k=0}^{\infty} e^{-\lambda
		k})^{-1}}{e^{\lambda}+1+e^{-\lambda}}
= \frac{1-e^{-\lambda}}{e^{\lambda}+1+e^{-\lambda}}
\eqdef \pesc.
\end{equation*}
Integration with respect to $\Prmp^\circ$ gives
\begin{equation*}
\pesc \leq \Prob^\circ(Y_n \not = \bfnull \text{ for all } n \geq 1) \leq 1.
\end{equation*}
Notice that the same bound holds when $\Prob^\circ$ is replaced by $\Prob$.
Thus
\begin{equation*}
\Prob^\circ\big(\tau_1^\mathrm{traps} \geq n | X_k \geq 1 \text{ for all } k \in \N\big)
\leq \frac1{\pesc} \Prob^\circ\big(\tau_1^\mathrm{traps} \geq n, X_k \geq 1 \text{ for all } k \in \N \big).
\end{equation*}
Analogously, when estimating $\Prob( \tau_1 \geq n )$, the time spent on the backbone can be neglected by Lemma \ref{Lem:moment estimate for tau_1^p},
so that it suffices to bound $\Prob(\tau_1^\mathrm{traps} \geq n)$ in this case.
We shall only estimate $\Prob^\circ(\tau_1^\mathrm{traps} \geq n, X_k \geq 1 \text{ for all } k \in \N)$ as $\Prob(\tau_1^\mathrm{traps} \geq n)$
can be estimated similarly.
To this end, we consider $(Y_n)_{n\in\N_0}$ and $(Y_n^\mathrm{p})_{n\in\N_0}$ as constructed in Section \ref{section:coupling}.
Further, we use the family $T^{\mathrm{ann}}_{ij}$, $i\in\Z$, $j\in\N$
of random variables introduced in Lemma \ref{Lem:annealed tail estimate for single trap}.
By construction, the number of times $(Y_n)_{n\in\N_0}$ visits any node in $\omega$ which is not in the interior
of a trap piece can be bounded by the number of times
$(Y_n^\mathrm{p})_{n\in\N_0}$ visits the corresponding node in $\omegaprune$.
This holds in particular for all trap entrances.
By Lemma \ref{Lem:pruned regeneration points},
there exist regeneration points of $( Y_n^\mathrm{p} )_{n\in\N_0}$.
These also are regeneration points for $(Y_n)_{n\in\N_0}$.
We have
\begin{equation*} 
\Prob^\circ\big(\tau_1^\mathrm{traps} \geq n, X_k \geq 1 \text{ for all } k \in \N \big)
\leq	\Prob^\circ\bigg( \sum_{i=1}^T \sum_{j=1}^{V_i} T_{ij} \geq n \bigg)
\leq	\Prob^\circ\bigg( \sum_{i=1}^{\rho_1^\mathrm{p} } 
\sum_{j=1}^{\tau_1^\mathrm{p} } T^{\mathrm{ann}}_{ij} \geq n \bigg),
\end{equation*}
where $T$ is the number of traps in $[0,\rho_1)$,
$V_i$ is the number of visits to the $i$th trap, $T_{ij}$ is the time $(Y_n)_{n\in\N_0}$ spends
during the $j$th excursion into the $i$th trap,
and $(T^{\mathrm{ann}}_{ij})_{i,j \in \N}$ is a family of random variables independent of $(\omegaprune,(Y_n^\mathrm{p})_{n\in\N_0})$
such that the $T^{\mathrm{ann}}_{ij}$, $i,j \in \N$ are independent given the family $(L_i)_{i \in \N}$
with $T^{\mathrm{ann}}_{ij}$ being distributed as the duration of one excursion of $(Y_n)_{n\in\N_0}$ under $P_{\omega,\lambda}$
into a trap of length $L_i$.
Since $(\rho_1^\mathrm{p}, \tau_1^\mathrm{p})$
and $(T^{\mathrm{ann}}_{ij})_{i,j\in\N}$ are independent, we can write this as
\begin{align}
\Prob^\circ\bigg( \sum_{i=1}^{ \rho_1^\mathrm{p} }\sum_{j=1}^{ \tau_1^\mathrm{p} }T^{\mathrm{ann}}_{ij} \geq n \bigg)
&= \sum_{k=1}^\infty \sum_{l=1}^\infty
\Prob^\circ \bigg(\rho_1^\mathrm{p} = k,\tau_1^\mathrm{p} = l, \sum_{i=1}^k \sum_{j=1}^l T^{\mathrm{ann}}_{ij} \geq n \bigg) \notag\\
&= \sum_{k=1}^\infty \sum_{l=1}^\infty 
\Prob^\circ \big(\rho_1^\mathrm{p} = k, \tau_1^\mathrm{p} = l \big)
\cdot \Prob\bigg( \sum_{i=1}^k \sum_{j=1}^l T^{\mathrm{ann}}_{ij} \geq n \bigg).	\label{eqn:independence Y^p and T_ij}
\end{align}
First look at $\Prob( \sum_{j=1}^l T^{\mathrm{ann}}_{ij} \geq n)$
for fixed $i$ and $l\in\N$. We write this as
\begin{equation*} 
\Prob\bigg( \sum_{j=1}^l T^{\mathrm{ann}}_{ij} \geq n \bigg)
= \Prob\bigg( \sum_{j=1}^l T^{\mathrm{ann}}_{ij} \geq n, \ell_i < m_0\vee m_1 \bigg)
+ \Prob\bigg( \sum_{j=1}^l T^{\mathrm{ann}}_{ij} \geq n,\ell_i \geq m_0\vee m_1 \bigg),
\end{equation*}
with $m_0,m_1$ as in Lemma \ref{Lem:annealed tail estimate for single trap}.
With $P_{m,\lambda}$ and $T_{ij}^\mathrm{qu}$, $i,j \in \N$ as in Lemma \ref{Lem:quenched tail estimate for single trap - lazy},
Markov's inequality and the convexity of $x\mapsto x^{\alpha+1}$ on
$[0,\infty)$ give
\begin{align*}
\Prob\bigg( \sum_{j=1}^l T^{\mathrm{ann}}_{ij} \geq n, \ell_i < m_0\vee m_1 \bigg)
&= \sum_{m=1}^{m_0\vee m_1 -1} \Prmp ( \ell_i = m ) 
P_{m,\lambda} \bigg( \sum_{j=1}^l T_{ij}^\mathrm{qu} \geq n \bigg) \\
&\leq (m_0\vee m_1) \underset{ m\in\{1,...,m_0\vee m_1-1\} }{\max}	
E_{m,\lambda} \bigg[ \bigg(\sum_{j=1}^l T_{ij}^\mathrm{qu} \bigg)^{\! \alpha+1} \bigg] \cdot n^{ -(\alpha+1) } \\
&\leq (m_0\vee m_1) \underset{ m\in\{1,...,m_0\vee m_1 -1\} }{\max} 
E_{m,\lambda} \bigg[ l^\alpha \sum_{j=1}^l (T_{ij}^\mathrm{qu})^{\alpha+1} \bigg] n^{-(\alpha+1)} \\
&= (m_0\vee m_1) l^{\alpha+1} n^{-(\alpha+1)} \underset{ m\in\{1,...,m_0\vee m_1 -1\} }{\max} 
E_{m,\lambda}\big[ (T_{i1}^\mathrm{qu})^{\alpha+1} \big].
\end{align*}
Let $N(k)$ be the number of times the walk $(S_n)_{n\in\N_0}$ visits vertex $k\in\{1,\ldots,m\}$.
Note that in order to describe $T_{i1}^\mathrm{qu}$, we also need to take lazy steps into account.
This means that, under $P_{m,\lambda}$, we have the following identity in law,
\begin{equation*}
T_{i1}^\mathrm{qu}	\disteq \sum_{k=1}^m \sum_{l=1}^{N(k)} (1+Z_{k,l})
\end{equation*}
where $N(k)$ has distribution $\text{geom}(\mathsf{e}_k)$ and the $Z_{k,l}$'s are a family of independent random variables, independent of $(N(1),\ldots,N(k))$,
with distribution $\text{geom}\big(\frac{e^\lambda+e^{-\lambda}}{e^\lambda+1+e^{-\lambda}}\big)$ for $k=1,\ldots,m-1$, $l\in\N$ and  $\text{geom}\big(\frac{e^{-\lambda}}{e^\lambda+1+e^{-\lambda}}\big)$ for $k=m$, $l\in\N$, respectively.
Since $m<m_0\vee m_1$ and the escape probability $\mathsf{e}_k$ is nonincreasing in $k$, we can bound $\mathsf{e}_k$ 
by $\mathsf{e}_{m_0 \vee m_1}$ for all $k\in\{1,\ldots,m\}$. We use this to stochastically bound $N(k)$.
In combination with the convexity of $x\mapsto x^{\alpha+1}$ on $[0,\infty)$ this leads to
\begin{align*}
E_{m,\lambda}\big[ (T_{i1}^\mathrm{qu})^{\alpha+1} \big]
&= E_{m,\lambda} \! \bigg[ \! \bigg(\sum_{k=1}^m \sum_{l=1}^{N(k)} (1\!+\!Z_{k,l}) \bigg)^{\!\!\alpha+1}\bigg]
\leq m^\alpha \sum_{k=1}^m E_{m,\lambda}\big[ N(k)^{\alpha+1} \big] E_{m,\lambda}\big[(1\!+\!Z_{k,m})^{\alpha+1}\big] \\
&\leq (m_0\vee m_1)^{\alpha+1} E_{m,\lambda}[N^{\alpha+1}] E_{m,\lambda}[(1\!+\!Z)^{\alpha+1}]
\end{align*}
where $N \sim \text{geom}(\mathsf{e}_{m_0\vee m_1})$ and $Z \sim \text{geom}\big(\frac{e^{-\lambda}}{e^\lambda+1+e^{-\lambda}}\big)$.
Thus
\begin{equation*} 
\underset{ m\in\{1,...,m_0\vee m_1 -1\} }{\max} 
E_{m,\lambda}\big[ (T_{i1}^\mathrm{qu})^{\alpha+1} \big]
\leq c(m_0,m_1,\lambda) 
= c(\lambda)
\end{equation*}
for some constant $c(\lambda)$. Combining this with the estimate for
$\sum_{j=1}^l T_{ij}^\mathrm{ann}$ in the case of traps of length larger or equal to $m_0\vee m_1$ from Lemma \ref{Lem:annealed tail estimate for single trap},
we get that there exists $d' = d'(p,\lambda)>0$ such that
\begin{equation*}
\Prob\bigg( \sum_{j=1}^l T^{\mathrm{ann}}_{ij} \geq n \bigg)	
\leq d' \Big( l^{\alpha+1} n^{-(\alpha+1)} + l^{\alpha+1} n^{-\alpha} + l e^{-\mu \frac{n}{6lr_\lambda}} \Big).
\end{equation*}
We further conclude
\begin{align}\notag
\Prob \bigg( \sum_{i=1}^k \sum_{j=1}^l T^{\mathrm{ann}}_{ij} \geq n \bigg) 
&\leq k \Prob\bigg( \sum_{j=1}^l T_{1j}^\mathrm{ann} \geq \frac{n}{k} \bigg) \\\notag
&\leq k d' \Big( l^{\alpha+1} \Big(\frac{n}{k}\Big)^{-(\alpha+1)} +  l^{\alpha+1} \Big(\frac{n}{k}\Big)^{-\alpha} + l e^{-\mu \frac{n}{6lr_\lambda k}}\Big) \\
&\leq k^{\alpha+2} l^{\alpha+1} d'\big(o(n^{-\alpha}) + n^{-\alpha} \big) + k l d'e^{-\mu \frac{n}{6lr_\lambda k}}. \label{eqn:annealed_sum}
\end{align}
Note that when estimating $\tau_1$ under $\Prob$, all calculations using Lemma
\ref{Lem:annealed tail estimate for single trap} involve an additional
factor of $\log n$. Combining \eqref{eqn:independence Y^p and T_ij} and
\eqref{eqn:annealed_sum}, we get
\begin{align}
\Prob^\circ\big( \tau_1^\mathrm{traps} \geq n \big)
&\leq d' \sum_{k,l=1}^\infty 
\Prob^\circ\big( \rho_1^\mathrm{p} = k, \tau_1^\mathrm{p} = l \big)
k^{\alpha+2} l^{\alpha+1} n^{-\alpha}(1+o_n(1)) \notag \\
&\hphantom{\leq} + d' \sum_{k,l=1}^\infty 
\Prob^\circ\big(\rho_1^\mathrm{p} = k, \tau_1^\mathrm{p} = l \big) k l e^{-\mu\frac{n}{6lr_\lambda k}}. \label{eqn:upper estimate subexp and exp parts}
\end{align}
For $k,l\in\N$, we write
\begin{equation*} 
\Prob^\circ\big( \rho_1^\mathrm{p} = k,
\tau_1^\mathrm{p} = l \big)
				=	\Prob^\circ\big( \tau_1^\mathrm{p} = l \big) \cdot
\Prob^\circ \big( \rho_1^\mathrm{p} = k
									| \tau_1^\mathrm{p} = l \big).
\end{equation*}
As the second factor vanishes for $k>l$, we get
\begin{align*}
\sum_{k,l=1}^\infty \Prob^\circ\big(\rho_1^\mathrm{p} = k, \tau_1^\mathrm{p} = l \big)
 k^{\alpha+2} l^{\alpha+1}
&= \sum_{l=1}^\infty \Prob^\circ\big(\tau_1^\mathrm{p} = l \big)
l^{\alpha+1} \sum_{k=1}^l \Prob^\circ\big( \rho_1^\mathrm{p} = k | \tau_1^\mathrm{p} = l \big) k^{\alpha+2}	\\
&\leq \sum_{l=1}^\infty \Prob^\circ\big(\tau_1^\mathrm{p} = l \big) l^{2\alpha+4}.
\end{align*}
Hence, it follows from Lemma \ref{Lem:moment estimate for tau_1^p} that the first sum on the right-hand side of \eqref{eqn:upper estimate subexp and exp parts} is bounded by a constant times $n^{-\alpha}$. 
For $\tau_1$ under $\Prob$, this becomes a constant times $n^{-\alpha} \log n$.
It also follows from Lemma \ref{Lem:moment estimate for tau_1^p} and Markov's inequality that, for any $\kappa>0$,
\begin{align*}
\sum_{k,l=1}^\infty \Prob^\circ\big(\rho_1^\mathrm{p} = k, \tau_1^\mathrm{p} = l\big) k l e^{-\mu \frac{n}{6lr_\lambda k}}
&= \sum_{l=1}^\infty \Prob^\circ(\tau_1^\mathrm{p}=l) l \sum_{k=1}^l \Prob^\circ \big(\rho_1^\mathrm{p} = k | \tau_1^\mathrm{p} = l \big) k  e^{-\mu \frac{n}{6lr_\lambda k}}	\\
&\leq \sum_{l=1}^\infty \Prob^\circ\big(\tau_1^\mathrm{p} = l \big) l^3 e^{-\mu \frac{n}{6l^2r_\lambda }}
\leq \E^\circ\big[ (\tau_1^{\mathrm{p}})^\kappa \big] \sum_{l=1}^\infty l^{-\kappa+3} e^{-\mu \frac{n}{6l^2r_\lambda }}.
\end{align*}
Setting $l^* \defeq \sqrt{\frac{\mu}{6r_\lambda(\alpha+1)}\frac{n}{\log n}}$ we get
\begin{align*}
\sum_{l=1}^\infty l^{-\kappa+3} e^{-\mu \frac{n}{6l^2r_\lambda }}
&= \sum_{l \leq l^*} l^{-\kappa+3} e^{-\mu \frac{n}{6l^2r_\lambda }} + \sum_{l>l^*} l^{-\kappa+3} e^{-\mu\frac{n}{6l^2r_\lambda }} \\
&\leq e^{-\mu\frac{n}{6r_\lambda(l^*)^2}} \sum_{l=1}^\infty l^{-\kappa+3} + (l^*)^{\frac{-\kappa+3}{2}} \sum_{l=1}^\infty l^{\frac{-\kappa+3}{2}}
= o(n^{-\alpha})
\end{align*}
for sufficiently large $\kappa$.
\end{proof}

\begin{appendix}

\section{Uniform integrability of renewal counting processes}	\label{sec:renewal theory}

In our proof of Theorem \ref{Thm:noncritical regimes},
we use that the suitably renormalised renewal counting process
of a delayed renewal process is uniformly integrable.
The following result is (more than) sufficient for our purposes.

\begin{proposition}	\label{Prop:uniform integrability of renewal counting
process}
Let $\xi_2,\xi_3,\ldots$ be a sequence of i.i.d.\ random variables
independent of $\xi_1$
such that $\Prm(\xi_k > 0) = 1$ for $k \in \N$, where $\Prm$ denotes the
underlying probability measure.
Suppose there are constants $d>0$ and $\alpha \in (1,2]$ such that
$\Prm(\xi_2 > t) \leq dt^{-\alpha}$ for all $t \geq 1$.
Then, with $\mu \defeq \Erm[\xi_2]$, $S_n \defeq \sum_{k=1}^n \xi_k$,
$\nu(t) \defeq \inf\{n \in \N: S_n > t\}$
and $a(t) \defeq t^{1/\alpha}$ if $\alpha \in (1,2)$ and $a(t) \defeq
\sqrt{t \log t}$ if $\alpha=2$, it holds that
\begin{equation}	\label{eq:nu(t)^+ u.i.}
\bigg(\exp\Big(\theta \frac{\nu(t)-t/\mu}{a(t)}\Big)\bigg)_{t \geq
2}	\quad	\text{is uniformly integrable for every } \theta > 0
\end{equation}
and
\begin{equation}	\label{eq:nu(t)^- u.i.}
\bigg(\Big(\frac{\nu(t)-t/\mu}{a(t)}\Big)_-^p\bigg)_{t \geq
2}	\quad	\text{is uniformly integrable for every } p \in (1,\alpha)
\end{equation}
for which there exists an $r > p$ with $\E[\xi_1^r]<\infty$.
\end{proposition}

The statements \eqref{eq:nu(t)^+ u.i.} and \eqref{eq:nu(t)^- u.i.}
have been shown in \cite{Iksanov+Marynych+Meiners:2016}
in the case where the $\xi_k$, $k \in \N$ are i.i.d.\
and $\xi_1$ is in the domain of attraction of an $\alpha$-stable law.
Unfortunately, we have not been able to apply a coupling argument
in order to deduce uniform integrability here from the main results in the
cited source.
However, the proofs given in \cite{Iksanov+Marynych+Meiners:2016} apply.
We shall provide a sketch of these proofs with the necessary changes needed
here.

\begin{proof}[Sketch of the proof of Proposition \ref{Prop:uniform
integrability of renewal counting process}]
Let $\theta > 0$, and denote by $\psi$ and $\varphi$
the Laplace transforms of $\xi_1$ and $\xi_2$, respectively,
i.e., $\psi(\lambda) = \Erm[\exp(-\lambda \xi_1)]$ and $\varphi(\lambda)
\defeq \Erm[\exp(-\lambda \xi_2)]$ for $\lambda \geq 0$.
Arguing as in (2.2) of \cite{Iksanov+Marynych+Meiners:2016}, we infer
\begin{align*}
\Erm\Big[\exp\Big(\theta \frac{\nu(t)-t/\mu}{a(t)}\Big)\Big]
&\leq		1 + 	\frac{\psi(\lambda)}{\varphi(\lambda)} \big( e^{\lambda \mu}
\varphi(\lambda) \big)^{\frac t\mu} \int_0^\infty e^x \varphi(\lambda)^{\frac{ x a(t) }{\theta}-1}\, \dx
\end{align*}
where the difference to (2.2) in \cite{Iksanov+Marynych+Meiners:2016} is a
factor $\psi(\lambda)/\varphi(\lambda)$,
which appears here since we allow the first step to have a different law
than the other steps.
Equation (2.7) in \cite[XIII.2]{Feller:1966} and Proposition
\ref{Prop:tail estimate} give
\begin{equation*}
\varphi(\lambda) 	=	1 - \mu \lambda + \lambda \int_0^\infty \!\! (1 -
e^{-\lambda x} ) \Prm(\xi_2 > x) \, \dx
\leq 1 - \mu \lambda + \int_0^\infty \!\! (1-e^{\lambda x}) (1 \wedge
dx^{-\alpha}) \, \dx.
\end{equation*}
The third summand on the right hand side is the second-order term of the
Laplace transform of a random variable with tail probability $1 \wedge d
x^{-\alpha}$ for $x > 0$.
From \cite[Theorem 8.1.6]{Bingham+Goldie+Teugels:1989}, we thus infer that
it is $\mathcal{O}(\lambda^\alpha)$ as $\lambda \to 0$ if $\alpha \in
(1,2)$
and $\mathcal{O}(\lambda^2 |\log \lambda|)$ if $\alpha=2$.
Choosing $\lambda^* \defeq \lambda/ a(t)$, this gives
\begin{align*}
e^{\lambda^* \mu} \varphi(\lambda^*)
&\leq	 \Big(\!1 + \frac{ \mu \lambda}{a(t)} + \mathcal{O}\big(t^{-\frac
2\alpha}\big) \! \Big)
		\bigg(\! 1 - \frac{ \mu \lambda}{a(t)}
		+ \frac{\lambda}{a(t)} \int_0^\infty \!\!\! \big(1 \! - \!
e^{-\frac{\lambda x}{a(t)}} \big) (1 \wedge d x^{-\alpha}) \, \dx  \!
\bigg)
= 1 + \mathcal{O}(t^{-1}),
\end{align*}
thus
\begin{equation*}
\sup_{t \geq 2} \big( e^{\lambda^* \mu} \varphi(\lambda^*)
\big)^{t/\mu}	<	\infty.
\end{equation*}
Further, the proof of (2.3) in \cite{Iksanov+Marynych+Meiners:2016}
applies and gives
\begin{equation*}
\sup_{t \geq t_0} \int_0^\infty e^x \varphi( \lambda^* )^{\frac{x
a(t)}{\theta} -1} \, \dx < \infty
\end{equation*}
for $t_0$ and $\lambda$ sufficiently large.
Uniform integrability of $(\exp( \theta \frac{\nu(t) - t/\mu}{a(t)}))_{t
\geq 2}$ now follows
from the Vall\'ee-Poussin criterion.

Turning to the second assertion, pick $1<p<\alpha$ and $r \in(p,\alpha)$
such that $\E[\xi_1^r]<\infty$.
Following the proof of (2.5) in \cite{Iksanov+Marynych+Meiners:2016}
with mild adaptions, we obtain
\begin{align*}
\Erm\big[ ( \nu\big(\E[S_n]\big) - n)_-^r \big] \leq r + \mathrm{const}
\cdot \Erm[|S_n - \E[S_n]|^r] = \mathcal{O}\big(a(n)^r \big)
\end{align*}
as $n\to\infty$.
Here, the last step follows from
\begin{align*}
\Erm[|S_n-n\mu|^r]
\leq
2^{r-1} \big( \Erm[|S_1-\mu|^r] + \Erm[|S_n-S_1 - (n-1)\mu|^r]\big).
\end{align*}
By assumption, $\Erm[S_1^r] = \Erm[\xi_1^r] < \infty$.
Further, positive and negative part of $\xi_2-\mu$
can be stochastically dominated by a nonnegative random variable with
tails of order $x^{-\alpha}$.
Hence it follows from \cite[Lemma 5.2.2]{Ibragimov+Linnik:1971} that
\begin{equation*}
\Erm[|S_n-S_1-(n-1)\mu|^r|] = \mathcal{O}(a(n)^r)	\quad	\text{as }	n \to
\infty.
\end{equation*}
The rest of the proof is as in \cite{Iksanov+Marynych+Meiners:2016}.
\end{proof}

\end{appendix}

\section*{Acknowledgements}
We would like to thank Volker Betz for many helpful discussions. 
Further, we are grateful to Alexander Marynych for pointing us to reference \cite{Flajolet+Sedgewick:1995}.

\bibliographystyle{plain}
\bibliography{RWRE}

\end{document}